\documentclass[11pt]{article}
\usepackage{amsthm}
\usepackage{amsmath}
\usepackage{amsfonts, amssymb, stmaryrd}
\usepackage[all]{xy}

\usepackage[usenames]{color}
\usepackage{cite}
\usepackage{fullpage}

\usepackage{hyperref}

\numberwithin{equation}{subsection}

\swapnumbers  

\newtheorem{theorem}{Theorem}[subsection]
\newtheorem{proposition}[theorem]{Proposition}
\newtheorem{corollary}[theorem]{Corollary}
\newtheorem{lemma}[theorem]{Lemma}
\theoremstyle{definition}
\newtheorem{definition}[theorem]{Definition}

\newtheorem{remark}[theorem]{Remark}
\newtheorem{example}[theorem]{Example}
\newtheorem{notation}[theorem]{Notation}
\newtheorem {empt}[theorem]      {}

\newcommand{\R}{\mathbb{R}}

\newcommand{\PP}{\mathbb{P}}
\newcommand{\GG}{\mathbb{G}}

\newcommand{\be}{\begin{enumerate}}
\newcommand{\bi}{\begin{itemize}}
\newcommand{\ee}{\end{enumerate}}
\newcommand{\ei}{\end{itemize}}

\renewcommand{\P}{\mathbb{P}}

\newcommand{\cP}{{\mathcal P}}

\renewcommand{\P}{\mathbb{P}}
\newcommand{\HH}{\mathbb{H}}

\newcommand{\V}{\mathbb{V}}

\newcommand{\Vect}{\mathsf{Vect}}

\newcommand{\Man} {\mathsf{Man}}

\newcommand{\Graph}{\mathsf{Graph}}

\newcommand{\FinGraph}{\mathsf{FinGraph}}

\newcommand{\Ctrl}{\mathsf{Ctrl}}

\newcommand{\inv}{^{-1}}

\newcommand{\toto}{\rightrightarrows}
\newcommand{\Hom}{\mathrm{Hom}}

\newcommand{\Aut}{\mathrm{Aut}}
\DeclareMathOperator{\essim}{essim}

\newcommand{\op}[1]{{#1}^{\mbox{\sf{\tiny{op}}}}}

\DeclareMathOperator{\rt}{rt}

\usepackage{mathrsfs} 
\DeclareMathAlphabet{\mathpzc}{OT1}{pzc}{m}{it}
\newcommand{\Control}{\mathsf{Control}}

\newcommand{\scI}{\mathscr{I}}

\renewcommand{\S}{\mathbb{{S}}}

\title{Modular dynamical systems on networks}

\author{Lee DeVille and Eugene Lerman}
\date{}

\begin{document}

\maketitle

\begin{abstract}

  We propose a new framework for the study of continuous time
  dynamical systems on networks.  We view such dynamical systems as
  collections of interacting control systems.  We show that a class
  of maps between graphs called {\sf graph fibrations} give rise to maps
  between dynamical systems on  networks.  This allows us to
  produce conjugacy between dynamical systems out of combinatorial
  data.  In particular we show that surjective graph fibrations lead
  to synchrony subspaces in networks.  The injective graph fibrations,
  on the other hand, give rise to surjective maps from large dynamical
  systems to smaller ones. One can view these surjections as a kind of
  ``fast/slow'' variable decompositions or as ``abstractions'' in the
   computer science sense of the word.

\end{abstract}

\tableofcontents

\section{Introduction}\label{sec:intro}

\subsection{Overview}

A fundamental question in the study of dynamical systems is to
determine the existence and properties of a map that intertwines the
dynamics of two different systems. Stated concretely, given two
manifolds $M,N$, and two flows $\varphi_t\colon M\to M, \psi_t\colon
N\to N$, does there exist a map $h\colon M\to N$ such that
\begin{equation}\label{eq:flow1}
  h\circ \varphi_t = \psi_t\circ h?
\end{equation}
Equivalently, given two manifolds and two vector fields $X\colon M\to
TM$, $Y\colon N\to TN$, does there exist a map $h$ such that
\begin{equation}\label{eq:vf1}
  dh\circ X = Y\circ h?
\end{equation}
If there is an $h$ that satisfies~\eqref{eq:flow1} or~\eqref{eq:vf1},
then we call this a {\em map of dynamical systems}.  Given such a map
$h$, we would like to understand its properties and to compute it
explicitly.

A common restriction requires that $h$ be invertible.  In this case,
it is said that we have exhibited a {\em conjugacy} between the two
dynamical systems, and this means that all of the dynamical features
of the flow are the same\footnote{Note that what we mean by ``the
  same'' depends on the category in which we work.  For instance if
  $h$ is a homeomorphism then we say that the dynamical systems are
  topologically the same.  Many of the implications of the existence
  of a conjugacy are worked out in~\cite{Smale.67, Alligood.etal.book,
    Devaney.book}.}.  The notion of conjugacy of dynamical systems
goes back at least to Poincar\'{e}~\cite{Poincare1, Poincare2,
  Poincare3}, it was further developed by Smale and
collaborators~\cite{Smale.67, Smale.Hirsch.book}, and is now the basic
notion in modern dynamical systems theory.

A more general notion of conjugacy arises from the relaxation of the
assumption of invertibility; here, the existence of the map $h$ still
produces significant information.
For example, the flow $\psi_t$ has a fixed point iff we can exhibit a
map $h\colon\{*\}\to N$, where $\{*\}$ is a one-point set, and $h$
satisfies~\eqref{eq:flow1}.
If $M = S^1$ and the flow $\varphi_t$ is given by $\varphi_t
(e^{i\theta}) = e^{2\pi it/T}e^{i\theta}$, then the existence
of $h\colon M\to N$ satisfying~\eqref{eq:flow1} amounts to the flow
$\psi_t$ having a periodic orbit of period $T$.

It is also common for $h$ to be chosen to be surjective. In this case
the map is typically termed a {\em
  semiconjugacy}~\cite{Brin.Stuck.book, Alligood.etal.book} and
certain nice properties follow.  We do not expand on this here, but we
will exploit the existence of semi-conjugacies for networked systems
below in Section~\ref{sec:surjective}.

The question of determining whether a map relating two dynamical
systems exists, and what its properties might be, is exceedingly
difficult~\cite{Yoccoz.95, Aubry.83, Mather.82, E.99} in general. In
some cases, even if such an $h$ is known to exist, determining its
form (or even qualitative properties) can be challenging.

In a different direction, dynamical systems defined on networks have
become the fundamental object of study across a variety of fields.
Some examples include the design of communications
networks~\cite{Rogers.Kincaid.book}; cognitive science, computational
neuroscience, and robotics (see, for example \cite{Knight.72, EK84,
  Kuramoto.91, Kuramoto.book, MacKay.book, Dayan.Abbott.book}); gene
regulatory networks~\cite{Bower.Bolouri.book, Shmulevich.etal.02,
  Alvarez-Buylla.etal.07} and more general complex biochemical
networks~\cite{Wilkinson.book}; and finally in complex active
media~\cite{Peskin.75, Tyson.Keener.88, Kapral.Showalter.book,
  Ermentrout.Rinzel.96, Keener.Sneyd.book}.

There are multiple definitions in the literature of what it means to
define a ``dynamical system on a network'' and we will not compare
them here. 
A common thread running through these definitions
is that imposing a network structure on a dynamical system 
should mean that
component $j$ of the system  depends upon
component $i$ of the system iff the underlying graph has an edge $i\to
j$.

In this paper we show that an imposition of a ``network structure'' on
a dynamical system allows us to produce maps between dynamical systems
in a precise, computable and combinatorial manner from finite data.
Thus the purpose of this manuscript is twofold: first, to present a
notion of a dynamical system ``consistent with a graph''; second, to
show that certain maps between graphs induce maps between the
dynamical systems that live on them.  In particular, we will show
below that all graph maps that respect a particular combinatorial
structure induce maps between the dynamical systems living on these
graphs.

 We focus  on the case where the dynamics is modeled by
vector fields on manifolds.  The interactions of subsystems are coded
by directed (multi-)graphs with ``labels''.  These labels in
particular, assign to each node of a graph the phase space of the
relevant subsystem.%
The ideas of the paper  can be
extended to both discrete-time, hybrid and stochastic systems, and we plan to
do so in future work.

As stated above, the main result of the paper is the construction of
maps of dynamical systems from maps of labeled graphs. %
In particular we show that 
surjective maps of graphs, such as the one
arising from quotienting a graph by an appropriate equivalence
relation, give rise to embeddings of dynamical systems; second,
injective maps of graphs give rise submersions of the corresponding
phase spaces and surjective maps of dynamical systems.  The former is
very useful in characterizing the ``modularity'' of a networked
dynamical system; the latter gives a precise mathematical formulation
of some intuitive notions of whether and how we can think of a large
dynamical system driven by a subsystem.

\subsection{Background and previous work}\label{sec:background}

The present paper  is inspired by 
several distinct bodies of work
that are well known in the applied mathematics communities.

The first body of work has been mainly applied to chemical reaction
systems, and, in some specific cases, to Petri nets; this work has
been used in both the deterministic and stochastic settings. 
One of the earliest
results in this direction is the ``zero deficiency theorem'' of
Feinberg~\cite{Feinberg.87, Feinberg.89, Schlosser.Feinberg.94,
  Feinberg.95, Craciun.Tang.Feinberg.06, Shinar.Alon.Feinberg.09},
first used to show the existence of stable equilibria in biochemical
systems and then expanded to statements about the existence and
structure of the equilibria in stochastic biochemical
systems~\cite{Anderson.Craciun.Kurtz.10}.  This type of methodology
has also been expanded to Petri nets~\cite{Gaubert.Gunawardena.98,
  Gunawardena.LN, Mairesse.Nguyen.09} and models describing gene
regulatory networks~\cite{Shmulevich.etal.02}.

The second body of work is due to Golubitsky, Stewart, and various
collaborators~\cite{Golubitsky.Stewart.84, Golubitsky.Stewart.85,
  Golubitsky.Stewart.86, Golubitsky.Stewart.86.2,
  Golubitsky.Stewart.87, Golubitsky.Stewart.Schaeffer.book,
  Field.Golubitsky.Stewart.91, Golubitsky.Stewart.Dionne.94,
  Dellnitz.Golubitsky.Hohmann.Stewart.95,
  Dionne.Golubitsky.Silber.Stewart.95, Dionne.Golubitsky.Stewart.96.1,
  Dionne.Golubitsky.Stewart.96.2, Golubitsky.Stewart.Buono.Collins.98,
  Golubitsky.Stewart.98, Golubitsky.Stewart.00,
  Golubitsky.Knobloch.Stewart.00, Golubitsky.Stewart.02,
  Golubitsky.Stewart.02.2, Stewart.Golubitsky.Pivato.03,
  Golubitsky.Nicol.Stewart.04, Golubitsky.Pivato.Stewart.04,
  Golubitsky.Stewart.05, Golubitsky.Stewart.Torok.05,
  Golubitsky.Stewart.06, Golubitsky.Josic.Brown.06,
  Golubitsky.Shiau.Stewart.07}. 
These authors considered a 
notion of ODEs (vector fields defined on Euclidean spaces) that were
consistent with a graph structure.  The resulting networks are called
coupled cell systems and the approach the groupoid formalism.
The main idea was to consider ``balanced'' equivalence relations on
the vertices of a graph.  They showed that these relations lead to the
existence of certain invariant subspaces termed ``polydiagonals.''  We
will see that the quotient maps resulting from balanced equivalence
relations are instances of graph fibration in the sense of Boldi and
Vigna~\cite{Vigna1}.  However, while we are greatly indebted to this
body of work for its intellectual inspiration, we also point out that
our approach differs from this work in some very specific ways.  We
elucidate the connections and contrasts in Remark~\ref{rem:GS}
below.

An alternative approach to coupled cell systems has been developed by
Field and collaborators~\cite{FieldCD,AgarwalField, Aguiar11}.  They
also considered ODEs and other types of dynamical systems consistent
with directed graphs. This approach is considered broadly equivalent
to that of Golubitsky {\em et al.}

\subsection{The contributions of the paper}
In this paper we propose a new framework for the study of continuous
time dynamical systems on networks.  We view such dynamical systems as
collections of interacting control systems.  We show that a class of
maps between graphs called {\sf graph fibrations} give rise to maps
between dynamical systems on networks.  This allows us to produce
conjugacy between dynamical systems out of combinatorial data.  While
the current work is certainly inspired by the methods and results of
both the Feinberg et al., Golubitsky et al. and, to a lesser extent, 
Field et al.\ groups, our approach and results differ in
several important 
respects.

\begin{enumerate}

\item Our basic philosophy is that of category theory --- so rather
  than study dynamical systems one at a time we aim to study  maps
  between all relevant dynamical systems  at once. 
  To quote Silverman~\cite{Silverman.book}:

\begin{quotation}
  A meta-mathematical principle is that one first studies (isomorphism
  classes of) objects, then one studies the maps between objects that
  preserve the objects' properties, then the maps themselves become
  objects for study and one tries to put a ``nice'' structure on the
  collection of maps (often modulo some equivalence relation).
\end{quotation}

\item Our set-up is coordinate-free and works for vector fields on
  manifolds, not just $\mathbb{R}^n$. In this case we are enlarging
  both the scope of the Feinberg et al. work (polynomial vector fields
  on positive orthants) and that of Golubitsky et al. (vector fields
  on Euclidean spaces).  This aspect of our approach is similar to the
  work of Field et al.
  There are several motivations for working on manifolds as opposed to
  ODEs living in Euclidean spaces.  These include dealing effectively
  with constraints, and
  extending 
 the results to the setting of geometrical mechanics.

\item It will be evident from the construction below that the quotient
  maps of graphs by balanced equivalence relations
  of~\cite{Golubitsky.Stewart.Torok.05} are special cases of graph
  fibrations --- they are the surjective graph fibrations. However,
  even in the case of surjective graph fibrations our maps of
  dynamical systems have the opposite direction from the maps in the
  groupoid formalism.  Rather than restricting from polydiagonals we
  extend from polydiagonals.  This allows us to deal with surjective
  and general graph fibration on the same footing.%

\end{enumerate}

\subsection{Motivating example}\label{sec:example}

Consider an ODE in $({\mathbb R}^n)^3$ of the form
\begin{equation}\label{eq1.2.1}
  \dot x_1 = f(x_2),\quad   \dot x_2 = f(x_1), \quad  \dot x_3 = f(x_2)
\end{equation}
for some smooth function $f: {\mathbb R}^n\to {\mathbb R}^n$.  That
is, consider the flow of the vector field
\[
F: ({\mathbb R}^n)^3\to ({\mathbb R}^n)^3, \quad 
F(x_1, x_2, x_3) = (f(x_2), f(x_1), f(x_2)).
\]
It is easy to check that $F$ is tangent to the diagonal 
\[
{\mathbb R}^n \simeq  \Delta 
= \{(x_1, x_2, x_3) \in (\mathbb{R}^n)^3 \mid x_1 = x_2 = x_3\}
\]
and that the restriction of the flow of $F$ to $\Delta$
is the flow of the ODE
\[
\dot u = f(u).
\] 
One can also see another invariant submanifold of $F$:
\[
({\mathbb R}^n)^2 \simeq  \Delta' = \{(x_1, x_2, x_3) 
\in ({\mathbb R}^n)^3 \mid x_1 = x_3\}.
\]
On $\Delta'$ the flow of $F$ is the flow of the ODE
\[
\dot v_1 = f(v_2), \quad \dot v_2 = f(v_1).
\]
Moreover the projection 
\[
\pi: ({\mathbb R}^n)^3 \to \Delta', \quad \pi (x_1, x_2, x_3) = (x_1, x_2, x_1)
\]
intertwines the flows of $F$ on $({\mathbb R}^n)^3$ and on $\Delta'$.  
We have thus observed two subsystems of
$(({\mathbb R}^n)^3, F)$ and three maps between the three dynamical systems:
\begin{equation}\label{eq1.2.2}
\xy
(-40, 0)*+{(\Delta, F|_\Delta)}="1";
(0,0)*+{(({\mathbb R}^n)^3, F)}="2";
(10,1)*+{}="2+"; 
(10,-1)*+{}="2-"; 
(40, 0)*+{(\Delta', F|_{\Delta'})}="3";
(30, 1)*+{}="3+";
(30, -1)*+{}="3-";
 {\ar@{^{(}->} "1";"2"};
 {\ar@{_{(}->} "3+";"2+"};
 {\ar@{->}_{\pi} "2-";"3-"}; 
\endxy
\end{equation}
Where do these subsystems and maps come from?  There is no obvious
symmetry of $({\mathbb R}^n)^3$ that preserves the vector field $F$
and fixes the diagonal $\Delta$ and thus could account for the
existence of this invariant submanifold.  Nor is there any
$F$-preserving symmetry that fixes $\Delta'$.  In fact the vector
field $F$ does not seem to have any symmetry.  The graph $G$
recording the interdependence of the variables $(x_1, x_2, x_3)$ in
the ODE \eqref{eq1.2.1} has three vertices and three arrows:
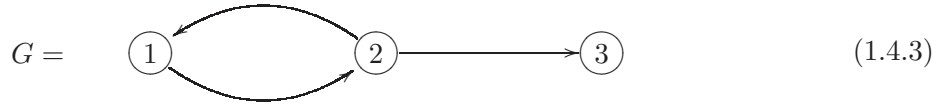
\begin{equation}\label{eq:favorite}
\xy
(-30,0)*+{G =}="n1";
(-15,0)*++[o][F]{1}="1"; 
(15,0)*++[o][F]{2}="2"; 
(45,0)*++[o][F]{3}="3"; 
{\ar@/_1.5pc/ "1";"2"};
{\ar@{->}_{} "2";"3"};
{\ar@/_1.5pc/ "2";"1"};
\endxy
\end{equation}
The graph has no non-trivial symmetries.  Nonetheless, the existence
of the subsystems $(\Delta, F|_\Delta)$, $(\Delta', F|_{\Delta'})$ and
the whole diagram of the dynamical systems \eqref{eq1.2.2} can be
deduced from certain properties of the graph $G$.  There are two
surjective maps of graphs:
\[
\varphi: G \to \xymatrix{ *+[o][F]{} \ar@(dr,ur)}
\]
and 
\[
\psi: G \to \xy
(-15,0)*++[o][F]{a}="1"; 
(15,0)*++[o][F]{b}="2"; 
{\ar@/_1.5pc/ "1";"2"};
{\ar@/_1.5pc/ "2";"1"};
\endxy,
\]
with $\psi $ defined on the vertices by $\psi(2) = b$, $\psi(1) = a =
\psi(3)$, and one embedding
\[
\tau: 
\xy 
(-10,0)*++[o][F]{a}="1"; 
(10,0)*++[o][F]{b}="2"; 
{\ar@/_1.2pc/ "1";"2"};
{\ar@/_1.2pc/ "2";"1"};
\endxy \quad
\hookrightarrow \quad 
\xy
(-10,0)*++[o][F]{1}="1"; 
(10,0)*++[o][F]{2}="2"; 
(30,0)*++[o][F]{3}="3"; 
{\ar@/_1.2pc/ "1";"2"};
{\ar@{->}_{} "2";"3"};
{\ar@/_1.2pc/ "2";"1"};
\endxy.
\]
We can collect all of these maps into one diagram
\begin{equation}\label{eq:Gdiag}
\xy
(-17, 0)*+{ \xymatrix{ *+[o][F]{} \ar@(dr,ur)}}="1";
(15,0)*+{ \xy
(-9,0)*++[o][F]{1}="1"; 
(9,0)*++[o][F]{2}="2"; 
(24,0)*++[o][F]{3}="3"; 
{\ar@/_1.2pc/ "1";"2"};
{\ar@{->}_{} "2";"3"};
{\ar@/_1.2pc/ "2";"1"};
\endxy}="2";
(68, 0)*+{\xy 
(-9,0)*++[o][F]{a}="1"; 
(9,0)*++[o][F]{b}="2"; 
{\ar@/_1.2pc/ "1";"2"};
{\ar@/_1.2pc/ "2";"1"};
\endxy }="3";
  {\ar@{->}_{\varphi} (-10,0); (-20, 0)};
 {\ar@{->}_{\tau} (50,1); (38,1)};
  {\ar@{->}_{\psi} (38,-1); (50, -1)}; 
\endxy
\end{equation}

A comparison of \eqref{eq1.2.2} and~\eqref{eq:Gdiag} evokes a  pattern:
for every map which intertwines dynamical systems in~\eqref{eq1.2.2},
there is a corresponding map of graphs in~\eqref{eq:Gdiag} with the
arrows reversed, and vice versa. \\

The same pattern holds when we replace the vector space $\mathbb{R}^n$
by an arbitrary manifold $M$.  Given a pair of manifolds $U$ and $N$, we
think of a map $X:U\times N\to TN$ with $X(u,n)\in T_nN$ as a control
system with the points of $U$ controlling the the dynamics on $N$.
Now consider a vector field 
\[
F:M^3 \to T(M^3) = TM\times TM \times TM
\]
of the form
\[
F(x_1, x_2, x_3) = (f(x_2, x_1), f(x_1, x_2), f(x_2, x_3))
\]
for some control system
\[
f: M\times M \to TM, \quad \textrm{with} \quad f(u,v)\in T_v M.
\]
Then once again the three maps of graphs in the
diagram~\eqref{eq:Gdiag} give rise to maps of dynamical systems
\begin{equation}\label{eq.EqDiag}
\xy
(-40, 0)*+{(\Delta_M, F|_{\Delta_M})}="1";
(0,0)*+{(M^3, F)}="2";
(10,1)*+{}="2+"; 
(10,-1)*+{}="2-"; 
(42, 0)*+{(\Delta'_M, F|_{\Delta'_M})}="3";
(30, 1)*+{}="3+";
(30, -1)*+{}="3-";
 {\ar@{^{(}->} "1";"2"};
 {\ar@{_{(}->} "3+";"2+"};
 {\ar@{->}_{\pi} "2-";"3-"}; 
\endxy
\end{equation}

\noindent
What accounts for the patterns we have seen? Notice that the dynamical
systems (\ref{eq.EqDiag}) are constructed out of {\em one} control
system $f:M\times M\to TM$.  At the same time, in each of the graphs
in (\ref{eq:Gdiag}), every vertex has exactly {\em one} incoming arc.
This is not a coincidence.  The rough idea for the technology which
generalizes this example is this: if we have a dynamical system made
up of repeated control system modules whose couplings are encoded in
graphs, then the appropriate maps of graphs lift to maps of dynamical
systems. Making this precise requires a number of constructions and
theorems; these make up the bulk of this paper.

\subsection{Main ideas of the paper}

We study {\em dependency} and {\em modularity} of networks and their effect on the concomitant dynamical systems.

By {\em dependency} we mean the following. 
Each node in a network corresponds to a single dynamical variable
living on a particular manifold. We then require that the variable
corresponding to node $i$ can depend on the variable corresponding to
node $j$ iff there is an edge in the graph from node $j$ to node $i$.
We give a more precise description of this requirement below and we
will denote the space all vector fields with this property by
$\S(G,\mathcal{P})$; see Section~\ref{sec:dependency} below for a
definition.

The rough idea of {\em modularity}
 is that
if we ever have multiple nodes of the graph that are ``the same'' and
have ``the same'' inputs, then we require that these nodes are
interchangeable in the dynamical system.  Speaking more precisely, we
will assume that in each network, each node has a ``type'' (it will,
in fact, be a manifold attached to this node which corresponds to the
phase space of the variables associated to that node), and if we ever
see two nodes, each with type $x$, with $n$ inputs in the graph, such
that these inputs are of type $x_1,\dots, x_n$, then 
the vector field defined on these two nodes must depend on their
inputs in exactly the same manner.  We will denote all vector fields
that respect this principle of modularity as $\V(G,\mathcal{P})$ and
give a precise definition of these vector fields in
Section~\ref{sec:modularity}.

In the example in the previous subsection, the principle of {\em
  dependency} tells us that the system living on the graph
in~\eqref{eq:favorite} 
must be of the form
\begin{equation}\label{eq:fgh}
  \dot x_1 = f(x_1,x_2),\quad \dot x_2 = g(x_2,x_1),\quad \dot x_3 = h(x_3,x_2),
\end{equation}
but $x_1$ cannot depend on $x_3$, for example.  The principle of {\em
  modularity} tells us that the functions $f,g,h$ must all be the
same, i.e. that
\begin{equation}\label{eq:fff}
  \dot x_1 = f(x_1,x_2),\quad \dot x_2 = f(x_2,x_1),\quad \dot x_3 = f(x_3,x_2),
\end{equation}
Therefore, all systems in $\S(G,\mathcal{P})$ must
satisfy~\eqref{eq:fgh}, but those in $\V(G,\mathcal{P})$ must satisfy
the stricter requirements~\eqref{eq:fff}.

\section{Networks and dynamics on networks}

The goal of this section is to define networks and dynamics on
networks in the context of continuous time dynamical systems.  It is
not uncommon to read that ``a network is a graph.'' This could not be
a complete story since by a network one usually means a collection of
interconnected subsystems.  We now make this precise.

\subsection{Graphs, manifolds and networks}

Throughout the paper {\em graphs} are directed multigraphs, possibly
with loops and multiple edges between nodes.  More precisely, we use
the following definition:
\begin{definition}\label{def:directed graph}
  A {\sf  graph} 
  $ G$ consists of two sets $ G_1$ (of arrows, or edges),
  $ G_0$ (of nodes, or vertices) and two maps $s, t:
   G_1 \to  G_0$ (source, target):
\[
 G = \{s,t\colon  G_1 \to  G_0\}.
\]
We write $ G = \{ G_1\toto G_0\}$. 
A graph $G$ is {\sf finite} if it has finitely many arrows and edges.
\end{definition}

\begin{definition}
  A {\sf map of graphs} $\varphi:A\to B$ from a graph $A$ to a graph
  $B$ is a pair of maps $\varphi_1:A_1\to B_1$, $\varphi_0:A_0\to B_0$
  taking edges of $A$ to edges of $B$, nodes of $A$ to nodes of $B$ so
  that for any edge $\gamma$ of $A$ we have
\[
\varphi_0 (s(\gamma)) = s(\varphi_1 (\gamma))\quad
\textrm{and} \quad \varphi_0 (t(\gamma)) = t(\varphi_1 (\gamma)).
\]
We will usually omit the indices 0 and 1 and write $\varphi(\gamma)$
for $\varphi_1 (\gamma)$ and $\varphi(a)$ for $\varphi_0 (a)$.
\end{definition}

\begin{remark}
  The collection of graphs and maps of graphs form a category
  $\Graph$.  The subcollection of finite graphs and maps of graphs
  forms a full subcategory $\FinGraph$.
\end{remark}

To construct a network from a graph we need to attach phase
spaces to its vertices.  Since we are interested in continuous time
dynamical systems, we choose phase spaces to be (finite dimensional
paracompact Hausdorff) manifolds.  Other choices, of course, may also
be reasonable, such as coordinate vector spaces $\R^n$ or manifolds
with corners.

\begin{definition}[Network]
  A {\sf network}  is a pair $(G, \cP)$ where $G$ is a
  finite graph and $\cP$ is a function that assigns to each node $a\in
  G_0$ of $G$ a manifold $\cP (a)$.  We refer to $\cP$ as a {\sf phase
    space function.}  Note that the target of the function $\cP$ is
  the collection $\Man$ of all manifolds: $\cP: G_0\to \Man$.

  A {\sf map of networks} from $(G,\cP)$ to $(G',\cP')$ is a map of
  graphs $\varphi:G\to G'$ so that
\[
\cP'\circ \varphi = \cP.
\]
\end{definition}
\begin{remark}
It is easy to see that composition of two maps of networks is again a
map of networks.  In other words networks  form a category.
We denote it by $\FinGraph/\Man$. 
\end{remark}

Given a network $(G,\cP)$ as defined above, a state of the network is
completely determined by the states of its nodes.  Hence the total
phase space of the network should be the product
\[
\PP (G,\cP) := \bigsqcap _{a\in G_0} \cP (a).
\]
Note, however, a small issue: there is no natural ordering on the
vertices of the graph $G$.  We could {\em choose} an ordering
$(a_1,\ldots, a_n)$ of the vertices and define the total phase space as
the Cartesian product
\[
\PP (G,\cP) := \prod _{i=1}^n \cP (a_i).
\]
However, it will be convenient {\em not} to choose an ordering of
vertices and use a slightly different notion of product.  This version
of the product is used, for example, in chemical reaction network
literature \cite{Gunawardena.LN}.

\begin{definition}\label{def:cat-prod}
 Given a family $\{M_s\}_{s\in S} $ of manifolds indexed by a finite
 set $S$, denote by $\bigsqcup _{s\in S}M_s$ their disjoint
 union\footnote{ The disjoint union may be defined by $\bigsqcup
   _{s\in S}M_s := \bigcup _{s\in S} (M_s\times \{s\})$.}.  The {\sf
   categorical product} of a finite family $\{M_s\}_{s\in S} $ of
 manifolds is the manifold
\[
\bigsqcap_{s\in S}M_{s} := \{x: S \to \bigsqcup _{s\in
  S}M_s \mid x(s)\in M_s \textrm{ for all }s\in S\}.
\]
\end{definition}

We note that for each index $s\in S$ we have projection maps $\pi_s:
\bigsqcap_{s'\in S}M_{s'}\to M_s$ are defined by
\[
\pi_s (x) = x(s).
\]
These projections are surjective submersions.

We denote $x(s)\in M_s$ by $x_s$ and think of it as $s^{th}$
``coordinate'' of an element $x\in \bigsqcap_{s\in S}M_{s}$.
Equivalently we may think of elements of the categorical product
$\bigsqcap_{s\in S}M_{s} $ as {\sf unordered} tuples $(x_s)_{s\in S}$
with $x_s\in M_s$.

\begin{remark}
  It is not hard to show that the categorical product, as defined
  above, has the following universal property: given a manifold $N$ and a
  family of smooth maps $\{f_s: N\to M_s\}_{s\in S}$ there is a unique
  map $f:N\to \bigsqcap_{s\in S}M_s$ with
\[
\pi_s \circ f = f_s \quad \textrm{ for all }s\in S.
\]
In fact categorical products are usually {\em defined} by this
universal property \cite{Awodey.book}.
\end{remark}

\begin{remark}
  For a family $\{M_s\}_{s\in S}$ of manifolds indexed by a finite set
  $S$ every ordering $\{s_1,\ldots, s_n\}$ of elements of $S$
identifies the categorical product $\bigsqcap _{s\in S} M_s$ (as a
manifold) with the Cartesian product $M_{s_1}\times \cdots \times
M_{s_n}$.
\end{remark}

We are now in position to state:
\begin{definition}[total phase space of a network  $(G,\cP)$]
  For a pair $(G,\cP)$ consisting of a graph $G$ and a phase space
  function $\cP$ we define the {\sf total phase space } of the network
$(G,\cP)$ to be the manifold
\[
\PP G\equiv \PP (G,\cP) := \bigsqcap _{a\in G_0} \cP (a),
\]
the categorical product of manifolds attached to the nodes of the
graph $G$ by the  phase space function $\cP$.%
\end{definition}

\begin{example}\label{example:1}
Consider the graph 
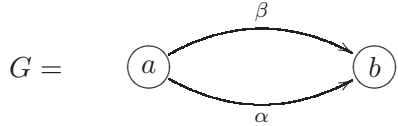
\[
\xy
(-30,0)*+{G =}="n1";
(-15,0)*++[o][F]{a}="1"; 
(15,0)*++[o][F]{b}="2"; 
{\ar@/_1.2pc/_{\alpha} "1";"2"};
{\ar@/^1.2pc/^{\beta} "1";"2"};
\endxy
\]
Define $\cP: G_0\to \Man$ by $\cP(a) = S^2$ (the two sphere) and
$\cP(b) = S^3$.  Then the total phase space $\PP (G, \cP)$ is the
Cartesian product $S^2 \times S^3$.
\end{example}

\begin{notation}
  If $G=\{\emptyset \toto \{a\}\}$ is a graph with one node $a$ and no
  arrows, we write $G=\{a\}$.  Then for any phase space function $\cP:
  G_0 =\{a\} \to \Man$ we abbreviate $\PP (\{\emptyset \toto \{a\} \},
  \cP:\{a\} \to \Man)= \PP (\{a\}, \cP:\{a\}\to \Man)$ as $\PP a$.
\end{notation}

\begin{proposition}\label{prop:phase-space-maps}
A map of networks $\varphi:(G,\cP)\to (G',\cP')$ naturally defines a map 
of corresponding total phase spaces
\[
\PP\varphi:  \PP G' \to \PP G.
\] 
\end{proposition}

\begin{proof}
  We use the universal property of the product $\PP G = \bigsqcap
  _{a\in G_0} \cP (a)$: to define a map $f$ from any manifold $N$ to
  $\PP G$ it is enough to define a family of maps $\{f_a: N\to \cP
  (a)\}_{a\in G_0}$.
For any node $a'$ of $G'$ we have the canonical projection 
\[
\pi'_{a'}:\PP G'\to \cP' (a').
\]
We therefore  define
\[
(\PP \varphi)_a := \pi'_{\varphi (a)}:\PP G' \to \cP' (\varphi(a) ) = \cP (a) 
\]
for all $a\in G_0$.
\end{proof}
\begin{example}\label{ex:1.0}
  Suppose $G$ is a graph with two nodes $a, b$ and no edges, $G'$ is a
  graph with one node $\{c\}$ and no edges, $\cP'(c)$ is a manifold
  $M$, $\varphi:G\to G'$ is the only possible map of graphs (it sends
  both nodes to $c$), and $\cP:G_0\to \Man$ is given by $\cP (a) = M =
  \cP(b)$ (so that $\cP'\circ \varphi = \cP$.  Then $\PP G'\simeq M$, 
\[
\PP G = \{ (x_a, x_b) \mid x_a\in \cP (a), x_b \in \cP (b)\} \simeq M\times M
\]
and $\PP\varphi: M\to M\times M$ is the unique map with $(\PP \varphi
(x))_a = x$ and $(\PP \varphi (x))_b = x$ for all $x\in \PP G'$.  Thus
$\PP \varphi: M\to M\times M$ is the diagonal map $x\mapsto (x,x)$.
\end{example}

\begin{example}
  Let $(G,\cP)$, $(G', \cP')$ be as in Example~\ref{ex:1.0} above and
  $\psi: (G',\cP')\to (G,\cP)$ be the map that sends the node $c$ to
  $a$.  Then $\PP \psi: \PP G \to \PP G'$ is the map that sends $(x_a,
  x_b)$ to $x_a$.
\end{example}

\begin{remark}
  It is not hard to show that if a map of networks $\varphi:(G,\cP)\to
  (G',\cP')$ is {\sf surjective} on vertices then $\PP \varphi:\PP G'
  \to \PP G$ is an {\sf embedding}.  If, on the other hand, $\varphi$
  is {\sf injective} on vertices, then $\PP \varphi$ is a {\sf
    surjective submersion}.
\end{remark}

\begin{remark}\label{rmrk:2.1.15}
  The total phase space map $\PP: \FinGraph/\Man \to \Man$ is a
  contravariant functor: given two maps of networks
\[
(G,\cP) \xrightarrow{\varphi} (G',\cP') \xrightarrow{\psi} (G'',. \cP'')
\]
we have 
\begin{equation}
\PP (\psi \circ \varphi) = \PP \varphi \circ \PP \psi.
\end{equation}
To indicate that $\PP$ reverses the direction of maps we 
now and subsequently  write
\[
\PP : \op{ (\FinGraph/\Man)} \to \Man.
\] 
The superscript $\op{}$ stands for the opposite category, i.e., the
category with the same objects but all the arrows reversed.
\end{remark}

\subsection{Open systems and their interconnections}

\noindent
Having set up a consistent way of assigning phase spaces to graphs
(that is, having set up networks of manifolds) we now take up a
construction of continuous time dynamical systems compatible with the
structure of the network.  We build vector fields on total phase
spaces of networks by interconnecting appropriate open systems.  Our
notion of interconnection is borrowed, to some extent, from the
control theory literature. See, for example, Willems~\cite{willems}.
We therefore start by recalling a definition of an open (control)
systems, which is essentially due to Brockett \cite{Brockett}.

\begin{definition}
A {\sf continuous time control system} (or an {\sf open system}) on a
manifold $M$ is a surjective submersion $p:Q\to M$ and a smooth map $
F\colon Q \to TM$ so that
\[
F(q) \in T_{p(q)} M
\]
 for
  all $q\in Q $.
 (cf., for example, \cite{Pappas}).
That is, the following diagram 
$
\xy (-10, 6)*+{Q}="1"; 
(6, 6)*+{TM} ="2"; 
(6,-3)*+{M}="3"; 
{\ar@{->}_{ p}  "1";"3"}; 
{\ar@{->}^{F} "1";"2"}; 
{\ar@{->}^{\pi} "2";"3"}; \endxy
$
commutes, where $\pi:TM \to M$ is the canonical projection.
\end{definition}

\begin{empt} Given a manifold $U$ of ``control
variables'' we may consider control systems of the form
\begin{equation}\label{eq:4}
F:M\times U\to TM.
\end{equation}
Here the submersion $p:M\times U \to M$ is given by
\[
p(x,u) = x.  
\]
The collection of all such control systems forms a vector
space that we denote by $\Ctrl (M\times U\to M)$:
\[
\Ctrl (M\times U\to M):= \{F:M\times U\to TM \mid F(x,u)\in T_x M\}.
\]

\end{empt}

\begin{notation}[Space of sections of a vector bundle]
Given a vector bundle $E\to M$ we denote the {\sf  space of sections} of $E\to M$ by $\Gamma E$ or by $\Gamma (E)$.
\end{notation}

Now suppose we are given a finite family $\{F_i:M_i \times U_i \to
TM_i\}_{i=1}^N$ of control systems and we want to somehow interconnect
them to obtain a closed system $\scI (F_1, \ldots, F_N)$.  This closed
system is a vector field on the product $\bigsqcap_i M_i$.  %
That is, 
$\scI (F_1, \ldots, F_N) $ is a section of the tangent bundle 
$T(\bigsqcap_i M_i)) \to \bigsqcap_i M_i$.
What additional data do we need to define the interconnection map
\[
\scI: \bigsqcap_i \Ctrl(M_i\times U_i \to M_i) \to 
\Gamma\, (T(\bigsqcap_i M_i))\quad ?
\]
The answer is given by the following proposition:

\begin{proposition} \label{prop:3.2} Given a family $\{p_j : M_j
  \times U_j \to M_j\}_{j=1}^N$ of projections and
  a family of smooth maps $\{s_j: \bigsqcap M_i \to M_j\times U_j\}$
  so that the diagrams
\[
\xy (-10, 6)*+{M_j \times U_j}="1"; 
(-10,-6)*+{\bigsqcap M_i}="3"; 
(10,-6)*+{ M_j}="4"; 
{\ar@{->}^{ p_j}  "1";"4"}; 
{\ar@{->}^{s_j} "3";"1"};
{\ar@{->}_{pr_j} "3";"4"};
 \endxy
\]
commute for each index $j$, there is an interconnection map $\scI$
making the diagrams
\[
\xy 
(-30, 6)*+{\bigsqcap_i \Ctrl(M_i\times U_i \to M_i)}="1"; 
(30, 6)*+{\Gamma
(T(\bigsqcap_i M_i))} ="2"; 
(-30,-6)*+{\Ctrl(M_j\times U_j \to M_j)}="3"; 
(30,-6)*+{\Ctrl (\bigsqcap_i M_i \xrightarrow {pr_j} M_j)}="4";
{\ar@{->}^{\scI} "1";"2"}; 
{\ar@{->}^{\varpi_j =D(pr_j)\circ -} "2";"4"};
{\ar@{->}^{\pi_j} "1";"3"};
{\ar@{->}_{\scI_j} "3";"4"};
 \endxy
\]
commute for each $j$.  The components $\scI_j$ of the
interconnection map $\scI$ are defined by $ \scI_j (F_j) := F_j \circ
s_j$ for all $j$.
\end{proposition}

\begin{proof}
The space of vector fields
$\Gamma (T(\bigsqcap_i M_i))$ on the product $\bigsqcap _i M_i$ is the product
of vector spaces $\Ctrl (\bigsqcap_i M_i \to M_j)$:
\[
\Gamma (T(\bigsqcap_i M_i)) = \bigsqcap _j \Ctrl (\bigsqcap_i M_i
\xrightarrow{pr_j}M_j).
\]
In other words a vector field $X $ on the product $\bigsqcap _i M_i$
is a tuple $X = (X_1, \ldots, X_N)$, where 
\[
 X_j := D (pr_j)\circ X.
\]
($D (pr_j):T \bigsqcap M_i \to TM_j)$ denotes the differential of the
canonical projection $pr_j: \bigsqcap M_i \to M_j$.)  Each component
$X_j : \bigsqcap _i M_i \to TM_i$ is a control system.

To define a map from a vector space into a product of
vector spaces it is enough to define a map into each of the factors.
We have canonical projections
\[
\pi_j: \bigsqcap_i \Ctrl(M_i\times U_i \to M_i) \to
\Ctrl(M_j\times U_j\to M_j), \quad j=1,\ldots, N.
\]
Consequently to define the interconnection map $\scI$ it is enough to
define the maps
\[
\scI_j: \Ctrl(M_j\times U_j\to M_j) \to \Ctrl (\bigsqcap_i M_i
\xrightarrow {pr_j} M_j).
\]
for each index $j$.
We therefore define the maps $\scI_j: \Ctrl(M_j\times U_j\to M_j)
\to \Ctrl (\bigsqcap_i M_i \xrightarrow {pr_j} M_j)$, $1\leq j\leq N$, by
\[
\scI_j (F_j) := F_j \circ s_j.
\]
\end{proof}

\begin{remark}\label{rmrk:2.21}
It will be useful for us to remember that the canonical projections
\[
\varpi_j: \Gamma (T\bigsqcap M_i))\to \Ctrl (\bigsqcap M_i \to M_j)
\]
are given by
\[
\varpi_j (X) = D (pr_j)\circ X,
\]
where $D (pr_j):T \bigsqcap M_i \to TM_j$ are the differential of the
canonical projections $pr_j: \bigsqcap M_i \to M_j$.
\end{remark}

\subsection{Interconnections and graphs}
We next explain how networks of manifolds give rise to interconnection
maps.  To do this precisely it is useful to have a notion of  {\em
  input trees} of a directed graph.  Given a graph, an input tree
$I(a)$ of a vertex $a$ is roughly, the vertex itself and all of the
arrows leading into it.  We want to think of this as a graph in its
own right, as follows.

\begin{definition}[Input tree]\label{def:input-tree}
  The {\em input tree} $I(a)$ at a vertex $a$ of a graph $ G$ is the
  graph with the set of vertices $I(a)_0$ given by
  \begin{equation*}
    I(a)_0 := \{a\}\sqcup t\inv (a);
  \end{equation*} 
  where, as before, the set $t\inv (a)$ is the set of arrows in $
  G$ with target $a$.  The set of edges $I(a)_1$ of the input tree is
  the set of pairs
  \begin{equation*}
    I(a)_1 := \{(a, \gamma)\mid \gamma \in G_1, \,\,\,t(\gamma) =a \},
  \end{equation*}
  and the source and target maps $I(a)_1\toto I(a)_0$ are defined by
  \begin{equation*}
    s(a, \gamma) = \gamma \quad\textrm{and}\quad  t(a, \gamma) = a.
  \end{equation*}
  In pictures,
  \[
  \xy  (-8,0)*++[o][F]{\gamma}="x";
  (8, 0)*++[o][F]{a}="z";
  {\ar@/^0.5pc/^{(a,\gamma) } "x";"z"};
  \endxy .
  \]
\end{definition}

\begin{example} \label{example:2}
Consider the graph $G=\xy
(-15,0)*++[o][F]{a}="1"; 
(15,0)*++[o][F]{b}="2"; 
{\ar@/_1.2pc/_{\alpha} "1";"2"};
{\ar@/^1.2pc/^{\beta} "1";"2"};
\endxy$ as in Example~\ref{example:1}.  Then the input tree $I(a)$ is
the graph with one node $a$ and no edges: 
\[
I(a) = \xy (0,0)*++[o][F]{a}\endxy.  
\]
The input tree $I(b)$ has three nodes and
two edges:
 \[
I(b) =   \xy  
(-8,4)*++[o][F]{\alpha }="1";
(-8,-4)*++[o][F]{\beta }="2";  
(8, 0)*++[o][F]{ b}="z";
  {\ar@/^0.5pc/^{(b,\alpha) } "1";"z"};
{\ar@/_0.5pc/_{(b,\beta) } "2";"z"};
  \endxy .
  \]
\end{example}

\begin{remark}\label{remark:xi}
 For each node $a$ of a graph $G$  we have a natural map of graphs 
\[
\xi= \xi_a : I(a)\to  G, \quad \xi_a (a, \gamma) = \gamma.
\] 
We stress that this need not be injective on nodes.  For instance in
Example~\ref{example:2} the map $\xi_b: I(b) \to G$ sends the {\em
  nodes} $\alpha$ and $\beta$ to the same node $a$.
\end{remark}

\begin{example}\label{exa:Ctrl}
Consider the graph

\begin{equation}\label{eq:four}
G =\xy
(-20,0)*++[o][F]{1}="1"; 
(0,0)*++[o][F]{2}="2"; 
(20,0)*++[o][F]{3}="3"; 
(40,0)*++[o][F]{4}="4";
{\ar@{->}@/_1pc/^{\beta} "1";"2"};
{\ar@/^1pc/^{\alpha} "1";"2"};
{\ar@{->}^{\gamma} "2";"3"};
{\ar@{->}@/_1pc/^{\zeta} "3";"4"};
{\ar@/^1pc/^{\epsilon} "3";"4"};
{\ar@{->}@/_3pc/^{\delta} "1";"4"};
\endxy,
\end{equation}
The four input trees of $G$ are the graphs
\begin{equation*}
  \xy
(0,0)*++[o][F]{1}="1"; 
\endxy,\quad
\xy
(0,5)*++[o][F]{\alpha}="1a"; 
(0,-5)*++[o][F]{\beta}="1b"; 
(10,0)*++[o][F]{2}="2";
{\ar@{->} "1a";"2"};
{\ar@{->} "1b";"2"}; 
\endxy, \quad
\xy
(0,0)*++[o][F]{\gamma}="2"; 
(10,0)*++[o][F]{3}="3";
{\ar@{->} "2";"3"};
\endxy, \quad
\xy
(0,10)*++[o][F]{\epsilon}="3a"; 
(0,0)*++[o][F]{\zeta}="3b"; 
(0,-10)*++[o][F]{\delta}="1"; 
(10,0)*++[o][F]{4}="4";
{\ar@{->} "3a";"4"};
{\ar@{->} "3b";"4"}; 
{\ar@{->} "1";"4"}; 
\endxy.
\end{equation*}
\end{example}

\begin{remark}\label{rmrk:2.3.5}
  The input tree $I(a)$ of a graph $G$ is a directed tree of height 1:
  for any vertex $x$ of $I(a)$ with $x\not =a$ there is exactly one edge
  with source $x$ and target $a$. Also $a$ is the only vertex of
  $I(a)$ which is not a source of any edge (it's a {\sf root} of
  $I(a)$), and all the other vertices of $I(a)$ cannot be targets of
  any edges (they are {\sf leaves} of $I(a)$).
\end{remark}
\begin{remark}\label{rmrk:2.3.6}
It follows from Remark~\ref{rmrk:2.3.5} above that if $\varphi: I(a)
\to I(b)$ is an isomorphism of two input trees (these graphs may be
input trees of two different graphs) then necessarily $\varphi(a) = b$
\end{remark}

\begin{remark}\label{rmrk:2.3.7}
  Given a network $(G,\cP)$ and a map of graphs $\varphi: H\to G$ we
  get a map of networks
\[
\varphi: (H,\cP\circ \varphi)\to (G, \cP),
\]
hence a map of manifolds
\[
\PP \varphi: \PP G\to \PP H.
\]
\end{remark}

\begin{empt}\label{empt:2.24}
  Let $(G,\cP)$ be a network and
  let $a$ be a node of $G$.  Consider the graph $\{a\}$ with one
  node and no arrows. Denote the inclusion of $\{a\}$ in $G$ by
  $\iota_a$ and the inclusion into its input tree $I(a)$ by $j_a$.
  Then the diagram 
\[
\xy
(-8, 10)*+{ \{a\}} ="1"; 
(8, 10)*+{ I(a)} ="2"; 
(0, -2)*+{G}="3";
{\ar@{->}^{ j_a} "1";"2"};
{\ar@{->}_{ \iota_a} "1";"3"};
{\ar@{->}^{\xi_a} "2";"3"};
\endxy
\]
commutes.  By Remarks~\ref{rmrk:2.3.7} and \ref{rmrk:2.1.15} we have a
commuting diagram of maps of manifolds
\[
\xy
(-10, 10)*+{ \PP\{a\}} ="1"; 
(10, 10)*+{ \PP I(a)} ="2"; 
(0, -5)*+{\PP G}="3";
{\ar@{<-}^{ \PP j_a} "1";"2"};
{\ar@{<-}_{ \PP \iota_a} "1";"3"};
{\ar@{<-}^{\PP \xi_a} "2";"3"};
\endxy
\]
\end{empt}

\begin{empt}\label{empt:2.25}
  Let us now examine more closely the map $\PP j_a: \PP I(a)\to \PP a$
  in \ref{empt:2.24} above.  Since the set of nodes $I(a)_0$ of the input tree $I(a)$ is the disjoint union
\[
I(a)_0 = \{a\} \sqcup t\inv (a),
\]
and since $\xi_a (\gamma)$ is the source $ s (\gamma)$ for any $\gamma \in t\inv (a) \subset I(a)_0$,
we have
\[
\PP I(a) = \cP(a) \times \bigsqcap _{\gamma \in t\inv (a)} \cP (s(\gamma)).
\]
Since $j_a:\{a\} \to I(a)_0 = \{a\} \sqcup t\inv (a)$ is the inclusion, 
\[
\PP j_a : \PP I(a) \to \PP a
\]
is the projection 
\[
 \cP(a) \times \bigsqcap _{\gamma \in t\inv (a)} \cP (s(\gamma)) \to \cP(a).
\]
Similarly 
\[
\PP \iota_a: \PP G \to \PP a
\]
is the projection
\[
\bigsqcap _{b\in G_0} \cP (b) \to \cP(a).
\]
\end{empt}
Putting \ref{empt:2.24} and \ref{empt:2.25} together we get

\begin{proposition}\label{prop:2.26}
  For each node $a$ of a network $(G,\cP)$ the
  diagrams 
\[
\xy
(-40, 10)*+{ \PP I(a)} ="0"; 
(-10, 10)*+{ 
\cP (a)\times \bigsqcap _{\gamma \in t\inv (a)}\cP (s (\gamma))} ="1"; 
(30, 10)*+{ \cP(a)} ="2"; 
(-10, -10)*+{\bigsqcap_{b\in G_0} \cP(b)}="3";
{\ar@{=}^{ } "1";"0"};
{\ar@{->}^(.75){ \PP j_a} "1";"2"};
{\ar@{->}_{ \PP \iota_a} "3";"2"};
{\ar@{->}^{\PP \xi_a} "3";"1"};
\endxy
\] 
commute.
\end{proposition}

\begin{example}\label{example:3}
Suppose $G= 
\xy
(-15,0)*++[o][F]{a}="1"; 
(15,0)*++[o][F]{b}="2"; 
{\ar@/_1.2pc/_{\alpha} "1";"2"};
{\ar@/^1.2pc/^{\beta} "1";"2"};
\endxy$ is a graph as in Example~\ref{example:1} and suppose $\cP:G_0
\to \Man$ is a phase space function.  Then 
\[
\PP I(b) = \cP (a)
\times \cP (a) \times \cP (b),
\]
 $\PP j_b$ is the projection $\cP (a)
\times \cP (a) \times \cP (b) \to \cP (b)$ and 
\[
\Ctrl (\PP I(b) \to
\PP b) = \Ctrl (\cP (a) \times \cP (a) \times \cP (b) \to \cP
(b)).
\]  On the other hand $\PP I(a) = \cP (a)$, $\PP j_a: \cP(a)\to
\cP (a)$ is the identity map and 
\[
\Ctrl (\PP I(a)\to \PP a) = \Gamma (T\cP(a)), 
\]
the space of sections of the tangent bundle $T\cP(a)$, that is, the
space of vector fields on the manifold $\cP (a)$.
\end{example}

\subsection{Dependency}\label{sec:dependency}

Given a network $(G,\cP)$ we have a product of vector spaces 
\[
 \bigsqcap _{a\in G_0} \Ctrl( \PP I(a)\to \PP a).
\]
The elements of the product are unordered tuples of $(w_a)_{a\in G_0}$
of control systems (cf.\ \ref{def:cat-prod}).  We may think of them as
sections of the vector bundle
\begin{equation}
\Control(G,\cP):=\bigsqcup _{a\in  G_0}\Ctrl( \PP I(a)\to \PP a) \to G_0 
\end{equation}
over the vertices of $G$.  
This is the main reason for thinking of the collection of infinite
dimensional vector spaces $\{\Ctrl( \PP I(a)\to \PP a)\}_{a\in G_0}$
as a vector bundle over a finite set.  It will be convenient to have a
notation for the space of sections of the bundle $\Control(G,\cP) \to
G_0$.

\begin{definition} \label{def:S}
Let  $(G,\cP)$  be a network, as above.
We refer to the bundle $\Control(G,\cP)\to G_0$ as the {\sf control
  bundle} of the network $(G,\cP)$.  We call the sections $(w_a)_{a\in
  G_0}$ of the control bundle {\sf virtual vector fields} on the
network $(G,\cP)$. We denote the space of sections by $\S(G,\cP)$.
Thus
\[
\S(G,\cP) := \bigsqcap _{a\in G_0}\Ctrl( \PP I(a)\to \PP a).
\]
\end{definition}

We now argue that an application of the interconnection map $\scI:\S
(G,\cP) \to \Gamma T(\PP (G,\cP))$ turns these ``virtual vector
fields'' into actual vector fields on the total phase space $\PP
(G,\cP)$ of the network.  Indeed observe that
Propositions~\ref{prop:3.2} and \ref{prop:2.26} give us
\begin{theorem}\label{thm:3.8}
  For a  network $(G,\cP)$  there exists a natural
  interconnection map
\[
\scI: \bigsqcap_{a\in G_0} \Ctrl (\PP I(a)\to \PP a) \to \Gamma (T\PP G))
\]
with 
\[
\varpi_a \circ \scI ( (w_b)_{b\in G_0}) =  w_a \circ \PP j_a
\]
for all nodes $a\in G_0$.  Here $\varpi_a: \Gamma (T\PP G) \to
\Ctrl( \PP G_0 \xrightarrow{\PP \iota_a} \PP a)$ are the projection
maps defined by $\varpi_a (X) = D (\PP \iota_a)\circ X$
(q.v.~Remark~\ref{rmrk:2.21}).
\end{theorem}

\begin{example}\label{example:4}
Consider the graph $G= \xy
(-15,0)*++[o][F]{a}="1"; 
(15,0)*++[o][F]{b}="2"; 
{\ar@/_1.2pc/_{\alpha} "1";"2"};
{\ar@/^1.2pc/^{\beta} "1";"2"};
\endxy$ as in Example~\ref{example:1} with a phase space function
$\cP: G_0\to \Man$.  Then the vector field
\[
X = \scI (w_a, w_b): \cP (a)\times \cP (b) \to T\cP (a)\times T\cP (b)
\]
is of the form
\[
X(x,y) = (w_a (x), w_b (x,x, y)) \quad \textrm{ for all } (x,y) \in
\cP (a)\times \cP (b).
\]
If
$G= \xy
(-15,0)*++[o][F]{a}="1"; 
(0,0)*++[o][F]{b}="2";
 (15,0)*++[o][F]{c}="5";
{\ar@/_1.2pc/_{} "1";"2"};
{\ar@/^1.2pc/^{} "1";"2"};
{\ar@{->}^{} "2";"5"};
\endxy$ and $\cP :G_0\to \Man$ is a phase space function, 
then 
\[
\left(\scI (w_a, w_b, w_c)\right) (x,y,z) =
(w_a (x), w_b (x,x, y), w_c (y,z)) 
\]
for all $(w_a,w_b, w_c) \in \S (G,\cP)$ and all $(x,y,z)\in \cP
(a)\times \cP (b) \times \cP (c))$.
\end{example}

\section{Modularity} \label{sec:modularity}
\subsection{Symmetry groupoid of a network}
In this section we give one possible version of what it means for some
of the open subsystems in the tuple of the constituent subsystems
$(w_a)_{a\in G_0} \in \S(G,\cP)$ on a network $(G,\cP)$
to be ``the same.''  We start with a pair of examples.
\begin{example}\label{ex:3.1.1}
Consider the graph
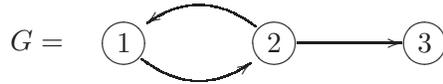
\[
G= \quad
\xy
(-10,0)*++[o][F]{1}="1"; 
(10,0)*++[o][F]{2}="2"; 
(30,0)*++[o][F]{3}="3"; 
{\ar@/_1.2pc/ "1";"2"};
{\ar@{->}_{} "2";"3"};
{\ar@/_1.2pc/ "2";"1"};
\endxy
\]
from subsection~\ref{sec:example}.  Choose a phase space function
$\cP:G_0 =\{1,2,3\} \to \Man$ with $\cP (1) = \cP (2) = \cP(3) =M$ for
some manifold $M$.  Then a typical tuple of open subsystems in
$\S(G,\cP)$ that defines the dynamics on the network is a triple of
the form $(f_1: M\times M\to TM, f_2: M\times M\to TM, f_3: M\times
M\to TM)$.  It make sense to require that $f_1 = f_2 =f_3$.  We can do it because the input trees $I(1)$, $I(2)$, $I(3)$ and the corresponding networks
$(I(i), \cP\circ \xi_i)$, $1\leq i\leq 3$, are all isomorphic (here,
as before $\xi_i: I(i)\to G$ are the canonical maps, q.v.\
Remark~\ref{remark:xi}). 
\end{example}

\begin{example}\label{ex:3.1.2}
Consider the graph
\[
G= \quad 
\xy
(-17,5)*++[o][F]{1}="1"; 
(-17,-5)*++[o][F]{2}="2"; 
(0,0)*++[o][F]{3}="3"; 
(20,0)*++[o][F]{4}="4";
{\ar@{->} "1";"3"};
{\ar@{->}^{} "2";"3"};
{\ar@{->}@/_1pc/^{} "3";"4"};
{\ar@/^1pc/^{} "3";"4"};
\endxy
\]
Again define a phase function $\cP$ by setting $\cP (i)$ to be the
same manifold $M$ for all $i$. An element of $\S(G,\cP)$ is then of
the form
\[
(f_1:M\to TM, f_2: M\to TM, f_3: M\times M \times M \to TM,
f_4:M\times M \times M \to TM).
\]
Now it does not make sense to require that $f_3 = f_1$ but it does
make sense to require that $f_1 = f_2$ and $f_3 = f_4$ (!).  Note that
in this example the networks $(I(1),\cP\circ \xi_1)$ and $(I(2),
\cP\circ \xi_2)$ are isomorphic as are the networks $(I(3),\cP\circ
\xi_3)$ and $(I(4), \cP\circ \xi_4)$.

If we were to set $\cP(1) = \cP (2) = \cP (3) = M$ and $\cP (4) = N\not = M$, then an element of   $\S(G,\cP)$ would be  of
the form
\[
(f_1:M\to TM, f_2: M\to TM, f_3: M\times M \times M \to TM,
f_4:M\times M \times N \to TN).
\] 
In this case setting $f_1 = f_2$ would make sense but setting $f_3 =
f_4$ would not. And while $(I(1),\cP\circ \xi_1)$ and $(I(2), \cP\circ
\xi_2)$ would still be isomorphic, the networks $(I(3),\cP\circ
\xi_3)$ and $(I(4), \cP\circ \xi_4)$ would not.
\end{example}

\begin{remark}\label{rmrk:3.1.3}
In Example~\ref{ex:3.1.1} there are $3^2$ isomorphisms 
\[
\varphi_{ij}: (I(j), \cP\circ \xi_j)\to (I(i), \cP\circ \xi_i), \quad
1\leq i,j\leq 3
\]
with
\[
\varphi_{ij}\circ \varphi_{jk} = \varphi_{ik}
\]
for all $i,j,k$, 
and with
\[
\varphi_{ii} = \textsf{id}
\]
for all $i$ (and consequently $\varphi_{ji} = \varphi_{ij}\inv$).
These 9 maps are an example of a groupoid.  
\end{remark}
We recall the shortest
definition of a groupoid: 
\begin{definition}
A {\sf groupoid} is a category with every
  morphism an isomorphism.  
\end{definition}

\begin{remark}
  One may think of a groupoid $\HH$ as a directed graph $\{\HH_1\toto
  H_0\}$ together with an associative multiplication  of pairs of edges with
  matched source and target:
\[
(a\xleftarrow{\alpha} b\xleftarrow{\beta} c) \mapsto
(a\xleftarrow{\alpha \beta }b)
\]
an inversion map
\[
 (a\xleftarrow{\alpha} b) \mapsto (a \xrightarrow{\alpha\inv } b)
 \]
 and a unit edge  $id_a: a\to a$ for every vertex $a$ of $\HH$.  We
 refer to the elements of $\HH_0$ as the {\em objects} of the groupoid
 $\HH$ and to the elements of $\HH_1$ as {\em isomorphisms} of $\HH$.
\end{remark}

\begin{example}
  In Remark~\ref{rmrk:3.1.3} the groupoid $\GG$ associated to the
  network $(G,\cP)$ has three objects, namely the networks $(I(i),
  \cP\circ \xi_i)$, $1\leq i \leq 3$, and 9 isomorphisms
  $\varphi_{ij}$, $1\leq i,j \leq 3$.  For the corresponding graph see
  \eqref{eq:3.1.1} below.
\end{example} 
  \begin{definition}[Symmetry groupoid $\GG (G,\cP)$ of a network] 
    The {\sf symmetry groupoid} $\GG = \GG
    (G,\cP)$ of a network $(G,\cP)$ is a category with the following
    sets of objects and isomorphisms, respectively. The set of objects
    $\GG_0$ of $\GG$ is the set of input networks 
\[
\{(I(a), \cP \circ \xi_a)\}_{a\in G_0}.
\]
The set of isomorphism $\GG_1$ of $\GG$ is
    the set of all possible isomorphisms of the input networks.
\end{definition}

\begin{example}\label{ex:3.1.19}
  Consider the network of Example~\ref{ex:3.1.1}.  As we have already
  pointed out the symmetry groupoid $\GG$ of this network has 3
  objects and 9 isomorphisms. It can be picture as follows:
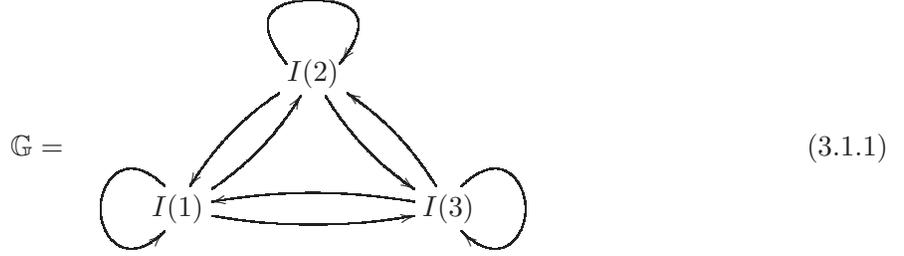
\begin{equation} \label{eq:3.1.1}
\GG = \quad
\xy
(-8,-8)*+{I(1)}="1"; 
(10,10)*+{I(2)}="2"; 
(28,-8)*+{I(3)}="3"; 
{\ar@/_.5pc/ "1";"2"};
{\ar@/_.50pc/ "2";"1"};
{\ar@/_.5pc/ "1";"3"};
{\ar@/_.5pc/ "3";"1"};
{\ar@/_.5pc/ "3";"2"};
{\ar@/_.5pc/ "2";"3"};
(34,-8)*{\xy
(3,3)*{}="A";
(3,-3)*{}="D";
"A"; "D" **\crv{(8,8) & (15,0)& (8,-8)}?(.99)*\dir{>}; 
\endxy};
(-14,-8)*{
\xy
(-3,3)*{}="A";
(-3,-3)*{}="D";
"A"; "D" **\crv{(-8,8) & (-15,0)& (-8,-8)}?(.99)*\dir{>}; 
\endxy
};
(10,15)*{
\xy
(-3.5,0)*{}="A";
(3.5, 0)*{}="D";
"A"; "D" **\crv{(-9,7) & (0,10)& (9,7)}?(.99)*\dir{>};
\endxy}
\endxy
\end{equation}
\end{example}

\subsection{Groupoid-invariant vector fields}
Given a network $(G, \cP)$ with a groupoid symmetry we should be able
to talk about invariant elements of the vector space $\S(G,\cP)$, the
vector space of constituent open subsystems.  This is indeed the case.
There are several ways of making sense of invariants.  The most
concrete cuts out the subspace of invariant by appropriate equations.
To set this up we need
a number of  short technical lemmas.  We formulate them in a generality
that is not needed immediately but will be useful later.  The point of
the lemmas is to prove that for a given a network $(G,\cP)$ there is a
natural action of its symmetry groupoid $\GG (G,\cP)$ on the vector
bundle $\Control(G,\cP)\to G_0$.

We start by spelling out what we mean by the action of $\GG (G,\cP)$
on $\Control(G, \cP)$.   
\begin{notation}
Denote the category of real vector spaces and linear maps by $\Vect$.
\end{notation} 
\begin{definition}\label{def:3.2.2}
  An action of the groupoid $\GG(G,\cP)$ on the vector bundle $\Control
  (G,\cP)$
is  a functor
\[
\rho: \GG(G,\cP) \to \Vect,
\]
so that
\[
\rho (I(a),\cP\circ \xi_a) = \Ctrl (\PP I(a) \to \PP a)
\]
for all nodes $a$ of the graph $G$. Here as above $\Vect$ denotes the
category of vector spaces and linear maps.
\end{definition}

\begin{remark}\label{rmrk:3.2.3}
The definition amounts to the following:
\begin{enumerate}
\item For any two vertices $a,b\in G_0$ and an isomorphism $\varphi:
  (I(a),\cP\circ \xi_a)\to (I(b),\cP\circ \xi_b)$ in the groupoid
  $\GG$, there is an isomorphism
\[
\rho (\varphi) : \Ctrl (\PP I(a)\to \PP a) \to \Ctrl (\PP I(b)\to \PP b);
\]
\item If $(I(a),\cP\circ \xi_a) \xrightarrow{\varphi} (I(b),\cP\circ
  \xi_b)\xrightarrow{\psi} (I(c),\cP\circ \xi_c)$ is a pair of
  isomorphism in $\GG$ then
\[
\rho (\psi \circ \varphi) = \rho (\psi) \circ \rho (\varphi);
\]
\item If $\varphi:I(a)\to I(a)$ is the identity isomorphism then
  $\rho(\varphi)$ is the identity linear map.
\end{enumerate}

The notion of an action of a groupoid on a vector bundle is fairly
old.  See, for example, \cite{Westman}.  In the case where the vector
bundle in question is a collection of vector spaces parameterized by a
finite set of objects of the groupoid, as is the case for
$\Control(G,\cP)$, it reduces to Definition~\ref{def:3.2.2} above.
\end{remark}

\begin{lemma}\label{lemma:3.1.11}
Suppose $\psi: (G, \cP) \to (G',\cP')$ is an isomorphism of networks. Then 
$\PP \psi :\PP G' \to \PP G$ is a diffeomorphism.
\end{lemma}
\begin{proof}
Let $\psi\inv$ denote the inverse of $\psi$.  Then
  $\psi\circ \psi\inv = \mathsf{id}_{ (G, \cP)}$ and $\psi\inv \circ
  \psi = \mathsf{id}_{ (G', \cP')}$.  Hence 
\[
\mathsf{id}_{\PP G} = \PP \mathsf{id}_{ (G, \cP)} = \PP (\psi\circ \psi\inv)
= \PP (\psi\inv) \circ \PP \psi. 
\]
By the same argument
\[
\mathsf{id}_{\PP G'}= \PP \psi \circ \PP (\psi\inv).
\]
Hence $\PP \psi$ is invertible with the inverse given by $\PP(
\psi\inv)$.
\end{proof}

\begin{remark} Here is a one line category-theoretic proof of
  Lemma~\ref{lemma:3.1.11}: since $\PP$ is a functor, it takes
  isomorphisms to isomorphisms.
\end{remark}

\begin{lemma}\label{lemma:3.1.10}
  Suppose $(G, \cP)$, $(G',\cP')$ are two networks, $\xi_a:I(a) \to G$
  the input tree of a vertex $a$ of $G$, $\xi_{a'}:I(a') \to G'$ the
  input tree of a vertex $a'$ of $G'$ and $\varphi: I(a)\to I(a')$ an
  isomorphism of trees with $\cP'\circ \xi_{a'}\circ \varphi = \cP
  \circ \xi_a$.  Then the linear map
\begin{equation}
\Ctrl(\varphi): \Ctrl (\PP I(a) \to \PP a) \to \Ctrl (\PP
I(a')\to \PP a')
\end{equation}
defined by
\begin{equation}\label{eq:3.1.4}
\Ctrl(\varphi): (F:\PP I(a) \to T\PP a)\mapsto 
D (\PP \varphi|_{\{a\}})\inv \circ F \circ \PP\varphi
\end{equation}
is an isomorphism.  Here $\varphi|_{\{a\}}: \{a\}\to \{a'\}$ is the
restriction of $\varphi$ to the subgraph $\{a\}$ of $G$ (by
Remark~\ref{rmrk:2.3.6} $\varphi$ has to send $a$ to $a'$); it is an
isomorphism.
\end{lemma}

\begin{proof}
  By assumption $\varphi: (I(a), \cP\circ \xi) \to (I(a'), \cP'\circ
  \xi')$ is an isomorphism of networks.  By Lemma~\ref{lemma:3.1.11},
  the maps $\PP \varphi$ and $\PP \varphi|_{\{a\}}$ are
  diffeomorphisms.  Therefore $\Ctrl(\varphi)$ has an inverse given
  by
\[
(F': \PP I(a') \to T\PP a')\mapsto D (\PP (\varphi|_{\{a\}})) \circ F
\circ \PP\varphi \inv.
\]
\end{proof}

It follows that we may define the functor $\rho:\GG (G, \cP) \to
\Vect$ on isomorphisms of the groupoid $\GG (G, \cP)$ by setting
\[
\rho (\varphi) : = \Ctrl (\varphi).
\]

\begin{lemma}\label{lemma:3.1.12}
  Given three networks $(G,\cP)$,$(G',\cP')$ and $(G'',\cP'')$, and a
  pair of isomorphism of input networks
\[
(I(a), \cP\circ \xi)\xrightarrow{\varphi}(I(a'), \cP'\circ
\xi'))\xrightarrow{\psi}(I(a''), \cP''\circ \xi'')
\]
we have
\begin{equation}
\Ctrl(\psi\circ \varphi) = \Ctrl{\psi}\circ \Ctrl{\varphi}.
\end{equation}
\end{lemma}

\begin{proof}
For $F\in \Ctrl(\PP I(a) \to \PP a)$ we have
\begin{eqnarray*}
  \Ctrl(\psi\circ \varphi)F&=&
  D (\PP (\psi\circ \varphi)|_{\{a\}}\inv) \circ F 
  \circ \PP (\psi\circ \varphi)\\
  &=& D ( (\PP \varphi|_{\{a\}}\circ \PP \psi|_{\{a'\}})\inv )\circ 
  F \circ \PP \varphi \circ \PP \psi \quad 
\textrm{ (since $\PP$ is a contravariant functor)}\\
  &=& D(\PP \psi|_{\{a'\}})\inv) \circ \left(D(\PP \varphi |_{\{a'\}})\inv) 
    \circ F \circ  \PP \varphi \right)\circ \PP \psi\\
  &=&\Ctrl(\psi) (\Ctrl (\varphi)F) .
\end{eqnarray*}
\end{proof}

We are now in the position to prove the main result of the section.
\begin{proposition}\label{prop:3.2.7}
The symmetry groupoid $\GG (G,\cP)$ of a network $(G,\cP)$ acts on the vector bundle $\Control(G,\cP)\to G_0$.  The action is given by 
\[
\left((I(a),\cP\circ \xi_a)\xrightarrow{\varphi} (I(b),\cP\circ
  \xi_b)\right)\mapsto \left(\Ctrl(\PP I(a) \to \PP
a)\xrightarrow{\Ctrl(\varphi)}\Ctrl(\PP I(b)\to \PP b)\right),
\]
where $\Ctrl(\varphi)$ is defined by \eqref{eq:3.1.4}.
\end{proposition}
\begin{proof} We need to check that the three conditions listed in
  Remark~\ref{rmrk:3.2.3} hold for $\rho(\varphi) = \Ctrl(\varphi)$.
The first one holds by Lemma~\ref{lemma:3.1.10}.  The second by Lemma~\ref{lemma:3.1.12}.
Note finally that by construction if $\varphi:I(a)\to I(a)$ is the
identity isomorphism then $\Ctrl (\varphi)$ is the identity linear
map.  We conclude that the functor
\[
\rho = \Ctrl: \GG(G,\cP) \to \Vect
\]
defines an action of the groupoid $\GG (G,\cP)$ on the vector bundle
$\Control(G, \cP)$.  
\end{proof}

Our next step is to define the space of invariant sections of the
vector bundle $\Control (G,\cP)\to G_0$ for this action, which is, by
definition, the space of invariant virtual vector fields on the
network.

\begin{definition}[Invariant virtual vector fields on a network]\label{def:VG}

Let $(G,\cP)$ be a network
  We  define the space $\V (G,\cP)$ of {\em
    groupoid-invariant  virtual vector fields} on the network to be 
\begin{eqnarray}\label{eq:defofVG}
 & \V G \equiv  \V(G,\cP) := \mbox{\hspace{13cm}} \\ 
&\{ (w_a)\in \S (G,\cP) \mid 
  \Ctrl(\sigma)w_a = w_b \textrm{ for all }  
\sigma\in \mathbb{G}(G,\cP) \mbox{ with } \sigma\colon I(a)\to I(b)\}. \nonumber
\end{eqnarray}
\end{definition}

\begin{example}
Consider the network of Example~\ref{ex:3.1.1}.  It is easy to see that
\[
\V (G,\cP) = \{(f_1,f_2, f_3) \in \S( G,\cP)\mid f_1 = f_2
= f_3\},
\]
where, as before $f_i: M\times M\to TM$ are control systems.  Note
that in this case the space of invariant virtual vector fields is
naturally isomorphic to the space $\Ctrl (\PP I(1) \to \PP 1) =
  \Ctrl(M\times M\to TM)$.  Note also that $\Ctrl(M\times M\to
  TM)$ is the space of invariant virtual vector field for the network
  $(G',\cP)$ where 
\[
G'=\quad \xymatrix{
  *+[o][F]{} \ar@(dr,ur)} 
\]
is the graph with one vertex and one edge and the function $\cP$
assigns the manifold $M$ to the one vertex of $G'$.
\end{example}

\begin{remark}
  The reader may wonder in what sense the sections in $\V (G,\cP)$ are
  ``invariant.''
  There are several ways to answer this question.  We start with the
  most concrete.  Note that the space $W^H$ of $H$-invariant vectors
  for a representation $\rho: H\to GL(W)$ of a {\em group} $H$
  satisfies
\begin{equation}\label{eq:3.2.5}
W^H =\{ w\in W\mid \rho (\sigma)w=w \textrm{ for all } \sigma \in H\}.
\end{equation}
It is easy to see now that \eqref{eq:defofVG} is an analogue of
\eqref{eq:3.2.5} for groupoids.  

More abstractly, we note that the
space $W^H$ is the limit of the functor $\rho:\underline{H} \to
\Vect$.  Here $\underline{H}$ denotes the category with one object $*$
and the set of morphisms $\Hom (*,*) = H$.  Similarly it is not hard
to see that $\V (G,\cP)$ as defined above by equation
\eqref{eq:defofVG} together with the evident projections $\V(G,\cP) \to \Ctrl
(\PP I(a) \to \PP a)$ is the limit of the functor $\Ctrl: \GG
(G,\cP)\to \Vect$.
\end{remark}

\begin{remark}
  We would like to think of the image $\scI(\V G)$ of the space $\V G$ of
  invariant virtual vector fields under the interconnection map $\scI:
  \S (G,\cP)\to \Gamma T\PP (G,\cP)$ as the space of
  ``groupoid-invariant vector fields'' on $\PP (G,\cP)$.  Note that
  this is {\em not} literally correct since there is no natural action
  of the groupoid $\GG (G,\cP)$ either on the the tangent bundle $T\PP
  (G,\cP)$ or on the space of its sections.
\end{remark}

\begin{remark}\label{rmrk:3.2.12}
  As we observed in Remark~\ref{rmrk:2.3.5} the graph underlying the
  input tree network $(I(a), \cP\circ \xi_a)$ of a network $(G,\cP)$
  is a directed tree of height 1.  If $\varphi: T_1\to T_2$ is an
  isomorphism of trees of height 1, then $\varphi$ necessarily sends
  the root $\rt T_1$ of the first tree to the root $\rt T_2$ of the
  second tree.  Hence if $\varphi: (T_1,\cP_1)\to (T_2,\cP_2)$ is an
  isomorphism of networks and $T_1, T_2$ are trees of height 1, it
  makes sense to define
\[
\Ctrl(\varphi) : \Ctrl(\PP T_1\to \PP \rt T_1) \to \Ctrl(\PP T_2\to
\PP \rt T_2)
\]
by a slight modification of \eqref{eq:3.1.4}:
\begin{equation}
\Ctrl (\varphi)F:= D\PP \left(\varphi|_{\rt T_1}\right)\inv \circ
F \circ \PP \varphi. 
\end{equation}
The proof of Proposition~\ref{prop:3.2.7} is then easy to modify to
show that that $\Ctrl$ is a well-defined functor from the groupoid of
height 1 tree networks and their isomorphism to the category $\Vect$
of (not necessarily finite dimensional) real vector spaces and linear
maps.
\end{remark}

\subsection{An alternative notion of modularity}

Throughout the paper we take the point of view that a network is a
directed graph $G$ together with an assignment of a phase space to
each vertex of $G$, that is, a pair 
\[
(G,\cP:G_0\to \textrm{collection
  of phase spaces}).
\]
  Golubitsky, Stewart and their collaborators in
their work on coupled cell networks additionally attach colors to
edges of graphs.  They require that maps of networks preserve the
colors.  In particular edges of input trees acquire colors from their
canonical maps into the defining graphs, and symmetry groupoids
consist of color preserving isomorphisms.  Thus from the point of view
of Golubitsky {\em et al.} we work with monochromatic graphs.  The
results of this paper do have their colored analogues.  The proofs,
{\em mutatis mutandis} are the same.  See \cite{deville2010dynamics}.

\section{Fibrations and invariant virtual vector fields}

We proved in Proposition~\ref{prop:phase-space-maps} that a map of
networks $\varphi:(G,\cP)\to (G',\cP')$ defines a smooth map $\PP
\varphi: \PP G'\to \PP G$ between their total phase spaces (going in
the opposite direction).  The map $\varphi$, in general, does not
induce a map between spaces of vector fields on the phase spaces $\PP
G$ and $\PP G'$. Nor does it induce a map between the spaces of
virtual vector fields $\S (G,\cP)$ and $\S (G',\cP')$, let alone the
spaces of groupoid-invariant virtual vector fields $\V G $ and $\V
G'$.  There is, however, a natural class of maps of networks that
does.  Following Boldi and Vigna~\cite{Vigna1} we call them {\em
  fibrations.} The notion of a graph fibrations is old.  It arouse
independently at different times in different areas of mathematics
under different names.  See \cite{VignaWP} for a discussion.

The goal of this section is to prove that a fibration of networks
$\varphi: (G,\cP)\to (G',\cP')$ naturally defines a linear map
$\varphi^*: \V (G',\cP')\to \V(G,\cP)$.  In the following sections we
show that the maps $\varphi^*$ and $\PP\varphi$ and the
interconnection maps of the two networks are compatible in the best
possible way. Consequently fibrations of networks give rise to maps of
dynamical systems.

\subsection{Fibrations}

\begin{definition}\label{def:fibration}
 A map $\varphi: G\to  G'$ of
  directed graphs is a {\sf fibration} if for any vertex $a$ of $ G$
  and any edge $e'$ of $ G'$ ending at $\varphi(a)$ there is a
  unique edge $e$ of $ G$ ending at $a$ with $\varphi (e) = e'$.

  A map of networks $\varphi: (G,\cP)\to (G',\cP')$ is a {\sf 
    fibration} if the corresponding map of graphs $\varphi:G\to G'$ is
  a fibration.
\end{definition}
\begin{example}\label{example:5}
The map of graphs
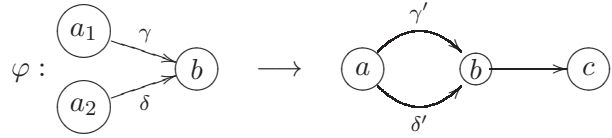
\[
\varphi: \xy
(-15,5)*++[o][F]{a_1}="1";
(-15,-5)*++[o][F]{a_2}="1'"; 
(0,0)*++[o][F]{b}="2"; 
{\ar@{->}^{\gamma} "1";"2"};
{\ar@{->}_{\delta} "1'";"2"};
\endxy
\quad \longrightarrow \quad
\xy
(-15,0)*++[o][F]{a}="1"; 
(0,0)*+[o][F]{b}="2";
 (15,0)*++[o][F]{c}="5";
{\ar@/_1.2pc/_{\delta'} "1";"2"};
{\ar@/^1.2pc/^{\gamma'} "1";"2"};
{\ar@{->}^{} "2";"5"};
\endxy
\]
sending the edge $\gamma$ to $\gamma'$ and the edge $\delta$ to
$\delta'$ is a graph fibration.
\end{example}
\begin{example}
  All the maps of graphs in \eqref{eq:Gdiag} are graph fibrations.  If
  we define the phase spaces functions on the three graphs by
  assigning to every node the same manifold $M$ then the corresponding
  maps of networks are  fibrations.
\end{example}
\begin{empt}\label{empt:4.2}
  Given any maps $\varphi:G\to G'$ of graphs and a node $a$ of $G$
  there is an induced map of input trees
\[
\varphi_a:I(a) \to I(\varphi(a)).
\]
On edges of $I(a)$ the map  is defined by 
\[
\varphi (a, \gamma):= (\varphi (a), \varphi (\gamma))
\]
(cf.\ Definition~\ref{def:input-tree}).
Moreover the diagram of graphs
\[
\xy
(-12, 8)*+{I(a)} ="1"; 
(12, 8)*+{I(\varphi(a))} ="2"; 
(-12, -10)*+{G}="3";
(12, -10)*+{G'}="4";
{\ar@{->}^{ \varphi_a} "1";"2"};
{\ar@{->}_{ \xi_a} "1";"3"};
{\ar@{->}^{ \xi_{\varphi(a)}} "2";"4"};
{\ar@{->}^{\varphi} "3";"4"};
\endxy
\]
commutes (the map $\xi_a: I(a)\to G$ from an input tree to the
original graph is defined in Remark~\ref{remark:xi}).
\end{empt}

\begin{lemma}\label{lemma:4.3}
If $\varphi:G\to G'$ is a graph fibration then the induced maps
\[
\varphi_a:I(a)\to I(\varphi(a))
\]
of input trees defined above are isomorphisms for all nodes $a$ of $G$.
\end{lemma}

\begin{proof}
  Given an edge $(\varphi(a), \gamma')$ of $I(\varphi(a))$ there is a
  unique edge $\gamma$ of $G$ with $\varphi(\gamma) = \gamma'$ and
  $t(\gamma) = a$. Consequently $\varphi_a (a,\gamma) =
  (\varphi(a),\gamma')$.  It follows that $\varphi_a$ is bijective on
  vertices and edges.
\end{proof}

\begin{corollary}\label{cor:4.1.6}
  If a map of networks $\varphi: (G,\cP)\to (G',\cP')$ is a 
  fibration then 
\[
\varphi_a: (I(a),\cP\circ \xi_a)\to (I(\varphi(a)),
  \cP'\circ \xi_{\varphi(a)})
\]
 is an isomorphism of networks.
\end{corollary}
\begin{proof}
  Follows immediately from Lemma~\ref{lemma:4.3} above and the
  definition of an isomorphism of networks.
\end{proof}

\subsection{Maps between spaces of invariant virtual vector fields}
The goal of this subsection is to show that fibrations of networks send
groupoid-invariant virtual vector fields to groupoid-invariant virtual
vector fields.  Namely we prove:
\begin{proposition} \label{prop:4.2.1}
A fibration
\[
\varphi: (G,\cP)\to (G',\cP')
\]
of networks  defines a linear map 
\[
\varphi^*: \S (G',\cP')\to \S(G,\cP)
\]
between spaces of sections of control bundles, that is, between spaces of
virtual vector fields on the networks in question.  

Moreover
$\varphi^*$ maps the space $\V (G',\cP')$ of groupoid-invariant
virtual vector fields to the space $\V(G,\cP)$.
\end{proposition}

\begin{proof}
  Recall (q.v.\ Definition~\ref{def:S}) that $\S (G,\cP) = \bigsqcap
  _{a\in G_0} \Ctrl( \PP I(a)\to \PP a)$.  We define
\[
\varphi^*:\bigsqcap _{a'\in G'_0} \Ctrl( \PP I(a')\to \PP a') \to
\bigsqcap _{a\in G_0} \Ctrl( \PP I(a)\to \PP a)
\]
by 
\[
(\varphi^*w')_a := \Ctrl(\varphi_a)\inv (w'_{\varphi(a)})
\]
for all $a\in G_0$.  Evidently $\varphi^* $ is linear.  

We now argue that invariant sections get mapped to invariant sections.
Consider $w'\in \V (G',\cP')$.  Let $\sigma:(I(a),\cP\circ \xi_a)\to
(I(b),\cP\circ \xi_b)$ be an isomorphism in the groupoid $\GG
(G,\cP)$. Since $\varphi$ is a fibration, the maps
\[
\varphi_a: (I(a),\cP\circ \xi_a) \to (I(\varphi(a)),\cP'\circ \xi_{\varphi(a)})
\]
and
\[
\varphi_b: (I(b),\cP\circ \xi_b) \to (I(\varphi(b)),\cP'\circ \xi_{\varphi(b)})
\]
are isomorphisms.  Therefore
\[ \varphi_b \circ \sigma \circ \varphi_a\inv:
(I(\varphi(a)),\cP'\circ \xi_{\varphi(a)}) \to
(I(\varphi(b)),\cP'\circ \xi_{\varphi(b)})
\] is an isomorphism of networks, hence an isomorphism in the groupoid
$\GG (G',\cP')$.  Since $w'$ is $\GG (G',\cP')$ invariant by
assumption, we have
\[ \Ctrl(\varphi_b \circ \sigma\circ \varphi_a\inv ) w'_{\varphi(a)} =
w'_{\varphi(b)}.
\] Since $\Ctrl$ is a functor on networks of height 1 trees (cf.\
Remark~\ref{rmrk:3.2.12}, it respects compositions and takes inverses
to inverses.  Consequently
\[ \Ctrl(\varphi_b) \circ \Ctrl(\sigma)\circ \Ctrl(\varphi_a)\inv
w'_{\varphi(a)} = w'_{\varphi(b)}.
\]
Thus
\[
\Ctrl(\sigma) (\varphi^* w')_a = \Ctrl(\sigma) \circ \Ctrl(\varphi_a)\inv
w'_{\varphi(a)} = \Ctrl(\varphi_b)\inv
w'_{\varphi(b)} = (\varphi^*w')_b,
\]
which proves that $\varphi^*w'\in \V (G,\cP)$.
\end{proof}

\begin{remark}
  The proof above shows that a fibration $\varphi: (G,\cP)\to
  (G',\cP')$ also induces a fully-faithful map of groupoids
\[
\GG(\varphi): \GG (G,\cP)\to \GG (G',\cP')
\]
which is given by 
\[
\left( (I(a),\cP\circ \xi_a)\xrightarrow{\sigma} (I(b),\cP\circ
  \xi_b)\right)\mapsto \left( (I(\varphi(a)),\cP\circ
  \xi_{\varphi(a)})\xrightarrow{\varphi_b \circ \sigma \circ \varphi_a\inv}
  (I(\varphi(b)),\cP\circ \xi_{\varphi(b)})\right).
\]
\end{remark}

\begin{remark} 
  Here is an alternative, more geometric, way to think of
  Proposition~\ref{prop:4.2.1} and its proof.  The collection of maps
\[
\left\{ \Ctrl(\varphi_a): \Ctrl(\PP I(a)\to \PP a)\to \Ctrl(\PP
    I(\varphi(a)) \to \PP\varphi(a)) \right\}
\]
define a map of vector bundles
\[
\tilde{\varphi}: \Control(G,\cP)\to \Control(G',\cP'),
\]
which restricts to an isomorphism on each fiber.  Hence $\Control
(G,\cP)\to G_0$ is the pullback of $\Control(G',\cP')\to G_0'$.  Consequently
sections of $\Control(G',\cP')\to G_0'$ pull back to sections of
$\Control(G,\cP)\to G_0$.  Moreover the vector bundle map
$\tilde{\varphi}$ intertwines the actions of the groupoids $\GG
(G,\cP)$ and $\GG (G'\cP'$. Hence invariant sections pull back to
invariant sections.
\end{remark}
\begin{remark} 
  In section~\ref{sec:6} below we show that somewhat surprisingly the
  map $\varphi^*:\V (G',\cP')\to \V(G,\cP) $ of
  Proposition~\ref{prop:4.2.1} is always surjective.  We also
  characterize the kernel of $\varphi^*$.  In particular if $\varphi:
  (G,\cP)\to (G',\cP')$ is a quotient map (in the setting of coupled
  cell networks the fibers of such $\varphi$ are equivalence classes
  of a balanced equivalence relation) then $\varphi^*:\V (G',\cP')\to
  \V(G,\cP) $ is an isomorphism.
\end{remark}

\subsection{Fibrations and maps of dynamical systems}

The goal of this subsection is to prove that fibrations of networks give rise
to maps between dynamical systems.  This  is arguably the main result of the
paper.
Here is a precise statement:

\begin{theorem}\label{thm:main}
  Let $\varphi: (G,\cP) \to (G',\cP')$ be a fibration of networks.
  Then for any groupoid-invariant virtual vector field $w'\in \V G'$ the map
  $\PP \varphi: \PP G'\to \PP G$ intertwines the vector fields $\scI'
  (w')$ and $\scI (\varphi^* w')$:
\begin{equation}\label{eq:main}
D (\PP \varphi) \circ \scI'  (w') = \scI (\varphi^* w') \circ \PP \varphi.
\end{equation}
Equivalently the diagram
\begin{equation}\label{eq:7}
\xy
 (-15, 10)*+{T\PP G'}="1";
 (15, 10)*+{T\PP G}="2";
 (-15, -5)*+{\PP G'}="3";
 (15, -5)*+{\PP G}="4";
{\ar@{->}^{D\PP\varphi } "1";"2"}; 
{\ar@{->}^{\scI'(w') } "3";"1"}; 
{\ar@{->}_{\scI (\varphi^*w') } "4";"2"};
{\ar@{->}_{\PP\varphi} "3";"4"};
\endxy
\end{equation}
commutes.
\end{theorem}

\begin{remark}
  Note that by Proposition~\ref{prop:4.2.1} since $w'$ is groupoid
  invariant virtual vector field on the network $(G',\cP')$, the
  pullback $\varphi^*w'$ is a {\em groupoid-invariant} virtual vector
  field on the network $(G, \cP)$, i.e., $\varphi^*w'\in \V (G,\cP)$.
\end{remark}

\begin{proof}[Proof of Theorem~\ref{thm:main}]
Recall that the manifold $\PP G$ is the product 
  $\bigsqcap_{a\in G_0} \PP a$.
  Hence the tangent bundle bundle $T\PP G$ is the product
  $\bigsqcap_{a\in G_0} T\PP a$. In particular for each node $a$ of the graph $G$, the canonical projection
\[
T\PP G\to T\PP a
\]
is the  differential of the map $\PP \iota_a: \PP G\to \PP a$. Here, as
before, $\iota_a: \{a\} \hookrightarrow G$ is the canonical inclusion
of graphs.  By the universal property of products, two maps into
$T\PP G$ are equal if and only if all their components are equal.
Therefore, in order to prove that \eqref{eq:7} commutes it is enough
to show that
\[
D \PP \iota_a \circ \scI (\varphi^*w') \circ \PP \varphi =
D \PP \iota_a \circ D \PP \varphi\circ \scI'(w')
\]
for all nodes $a\in G_0$.  By definition of the restriction
$\varphi|_{\{a\}}$ of $\varphi:G\to G'$ to $\{a\}\hookrightarrow G$,
the diagram
\begin{equation}\label{eq:**}
\xy
 (-10, 15)*+{\{a\}}="1";
 (10, 15)*+{\{\varphi(a)\}}="2";
 (-10, 0)*+{ G}="3";
 (10, 0)*+{ G'}="4";
{\ar@{->}^{\varphi|_{\{a\}} } "1";"2"}; 
{\ar@{->}_{\iota_a } "1";"3"}; 
{\ar@{->}_{\iota_{\varphi(a)} } "2";"4"};
{\ar@{->}_{\varphi} "3";"4"};  
\endxy
\end{equation}
commutes.  By the definition of the pullback map $\varphi^*$ and the
interconnection maps $\scI$, $\scI'$ the diagram
\begin{equation}\label{eq:8}
\xy
 (-15, 15)*+{T\PP a}="1";
 (15, 15)*+{T\PP\varphi(a)}="2";
 (-15, -0)*+{\PP I(a)}="3";
 (15, 0)*+{\PP I(\varphi(a)) }="4";
 (-15, -15)*+{\PP  G}="5";
(15, -15)*+{\PP  G'}="6";
{\ar@{->}^{D\PP\varphi|_{\{a\}} } "2";"1"}; 
{\ar@{->}_{(\varphi^*w')_a } "3";"1"}; 
{\ar@{->}^{w'_{\varphi(a)} } "4";"2"};
{\ar@{->}_{\PP\xi_a } "5";"3"}; 
{\ar@{->}^{\PP\xi_{\varphi(a)} } "6";"4"}; 
{\ar@{->}_{\PP\varphi_a} "4";"3"};
{\ar@{->}^{\PP\varphi} "6";"5"};
{\ar@/^{3pc}/^{\scI (\varphi^*w')_a} "5";"1"};
{\ar@/_{3pc}/_{\scI' (w')_{\varphi(a)}} "6";"2"};
\endxy
\end{equation}
commutes as well.  We now compute:
\begin{eqnarray*}
D\PP\iota_a \circ \scI (\varphi^*w')\circ \PP\varphi
  &= & (\scI (\varphi^*w'))_a \circ \PP\varphi 
\quad \textrm{ by definition of } \scI (\varphi^*w')_a \\
 &= & D \PP (\varphi|_{\{a\}})\circ \scI' (w')_{\varphi(a)}
  \quad \quad \quad \textrm{\quad by (\ref{eq:8})} \\
  & = & D \PP (\varphi|_{\{a\}})\circ D\PP \iota_{\varphi(a)} \circ \scI' (w')
 \quad \textrm{ by definition of } \scI' (w')_{\varphi(a)}\\
  &=& D \PP \left(\iota_{\varphi(a)} \circ \varphi|_{\{a\}}
\right) \circ \scI' (w')
  \quad \quad \textrm{ since $\PP$ is a contravariant functor }\\
  & = &  D \PP \left(\varphi \circ \iota_a 
\right) \circ \scI' (w')
 \quad  \quad \quad \quad  \quad\textrm{ by \eqref{eq:**} }\\
 &=& D \PP (\iota_a) \circ D\PP \varphi  
  \circ  \scI' (w').
\end{eqnarray*}
\end{proof}

\begin{remark}\label{rem:GS}
  In Lemma~\ref{lemma5.1.1} below we show that surjective fibrations
  of networks give rise to embeddings of dynamical systems.  Since
  balanced equivalence relations of the groupoid formalism of
  Golubitsky {\em et al.}~\cite{Golubitsky.Stewart.84,
    Golubitsky.Stewart.85, Golubitsky.Stewart.86,
    Golubitsky.Stewart.86.2, Golubitsky.Stewart.87,
    Golubitsky.Stewart.Schaeffer.book, Field.Golubitsky.Stewart.91,
    Golubitsky.Stewart.Dionne.94,
    Dellnitz.Golubitsky.Hohmann.Stewart.95,
    Dionne.Golubitsky.Silber.Stewart.95,
    Dionne.Golubitsky.Stewart.96.1, Dionne.Golubitsky.Stewart.96.2,
    Golubitsky.Stewart.Buono.Collins.98, Golubitsky.Stewart.98,
    Golubitsky.Stewart.00,
    Golubitsky.Knobloch.Stewart.00, Golubitsky.Stewart.02,
    Golubitsky.Stewart.02.2, Stewart.Golubitsky.Pivato.03,
    Golubitsky.Nicol.Stewart.04, Golubitsky.Pivato.Stewart.04,
    Golubitsky.Stewart.05, Golubitsky.Stewart.Torok.05,
    Golubitsky.Stewart.06, Golubitsky.Josic.Brown.06,
    Golubitsky.Shiau.Stewart.07} define quotient networks, each
  balanced equivalence relation give rise to a surjective maps of
  graphs and hence to surjective fibration of networks in our sense.
  Thus a special case of Theorem~\ref{thm:main} generalizes one
  direction of the groupoid formalism correspondence between invariant
  subspaces and balanced equivalence relations from ordinary
  differential equations to vector fields on manifolds. More
  specifically Theorem~\ref{thm:main} is a generalization, to
  manifolds, of Theorem 5.2 (direction (b))
  of~\cite{Golubitsky.Stewart.Torok.05} and of Theorem 9.2
  of~\cite{Stewart.Golubitsky.Pivato.03}.  We do not attempt to
  establish the converse.  More specifically, we do not attempt to characterize
  submanifolds of total phase spaces of networks that are preserved by
  all groupoid invariant  vector fields---we are only speaking to the ``forward'' direction.

\end{remark}
\section{Dynamical consequence of Theorem~\ref{thm:main}} 
\label{sec:surjective}

In this section, we will discuss the implications of
Theorem~\ref{thm:main}.  Consider a fibration $\varphi: (G,\cP)\to
(G',\cP')$ of networks.  Then $\varphi$ defines a map $\varphi_0:
G_0\to G_0'$ from the set of vertices of the graph $G$ to the set of
vertices of the graph $G'$.  In general $\varphi_0$ is neither
injective nor surjective.  However if a graph fibration $\varphi:G\to
G'$ is surjective on vertices, it is automatically surjective on
edges.  Similarly if a graph fibration $\varphi:G\to G'$ is injective
on vertices, then it is injective on edges as well.  From now on we
simply talk about injective and surjective graph fibrations.

Next observe that a given fibration $\varphi: (G,\cP)\to (G',\cP')$
can always be factored as a map onto its image followed by the
inclusion of the image:
\[
(G,\cP)\xrightarrow{\varphi} (\varphi(G),
\cP')\stackrel{i}{\hookrightarrow} (G', \cP')
\]
Hence any fibration can be factored as a surjection followed by an
injection.  We next analyze surjective and injective fibrations of
networks.

\subsection{Surjective fibrations}

\begin{lemma}\label{lemma5.1.1}
  Suppose $\varphi:(G,\cP)\to (G',\cP')$ is a surjective fibration.
  Then $\PP\varphi: \PP G'\to \PP G$ is an embedding whose image is a
  ``polydiagonal''
\[
\Delta_\varphi = \{x\in \P G\mid x_a = x_b\textrm{ whenever
}\varphi(a) = \varphi(b)\}.
\]
\end{lemma}
\begin{proof}
  Assume first for simplicity that $G'$ has only one vertex $*$ and
  $\cP'(*) = M$.  Then for any vertex $a$ of $G$ we have
\[
\cP(a) = \cP' (\varphi(a)) = \cP'(*) = M,
\]
$\PP G' = M$ and $\PP G = M^{G_0}$, where as before $G_0$ is the set
of vertices of the graph $G$.  By
Proposition~\ref{prop:phase-space-maps} the map $\PP\varphi:M\to
M^{G_0}$ is of the form
\[
\PP \varphi (x) = (x,\ldots, x)
\]
for all $x\in M$.

In general 
\[
\PP \varphi: \PP G' =\bigsqcap_{a'\in G'_0} \cP' (a') \to \bigsqcap_{a'\in G'_0} \left( \bigsqcap_{a\in \varphi\inv (a')} \cP' (a')\right) = \PP G
\]
is the product of maps of the form
\[
\cP' (a') \to \bigsqcap_{a\in \varphi\inv (a')} \cP' (a'), \quad
x\mapsto (x,\ldots, x).
\]
\end{proof}

\begin{example}
We consider the following surjective fibration $\varphi:(G,\cP)\to (G',\cP')$ of networks.  We take $G$ to be  the graph
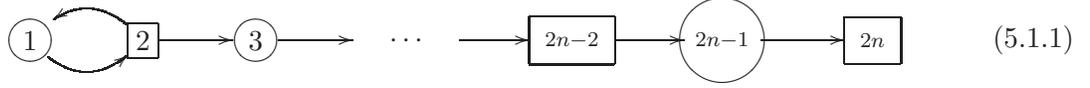
\begin{equation}\label{eq:xystring}
\xy
(-45,0)*++[o][F]{1}="1"; 
(-30,0)*+[F]{2}="2"; 
(-15,0)*++[o][F]{3}="3"; 
(0,0)*++{} ="4"; 
{\ar@{->}@/_1pc/ "2";"1"};
{\ar@/_1pc/ "1";"2"};
{\ar@{->}_{} "2";"3"};
{\ar@{->}_{} "3";"4"};
\endxy
\cdots
\xy
(-17,,0)*++{}="3"; 
(0,0)*++[F]{\scriptstyle{2n-2}}="4"; 
(20,0)*++[o][F]{\quad\scriptstyle{2n-1}\quad}="5"; 
(17,0)*++{}="5'"; 
(23,0)*++{}="5*";
(40,0)*++[F]{\scriptstyle{2n}}="6";
{\ar@{->}_{} "3";"4"};
{\ar@{->}_{} "4";"5'"};
{\ar@{->}_{} "5*";"6"};
\endxy
\end{equation}
with $2n$ vertices ($n\geq 2$). We choose a phase space function $\cP$
that assigns a manifold $M$ to all odd numbered vertices and a (different)
manifold $N$ to all even numbered vertices.  We take $G'$ to be the
graph
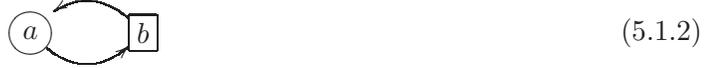
\begin{equation}\label{eq:cycle}
  \xy
(-45,0)*++[o][F]{a}="1"; 
(-30,0)*+[F]{b}="2"; 
{\ar@{->}@/_1pc/ "2";"1"};
{\ar@/_1pc/ "1";"2"};
\endxy
\end{equation}
with two vertices and two arrows.  We set $\cP'(a) =M$ and $\cP'(b) =
N$.  We define the surjective fibration $\varphi\colon G\to G'$ by
setting
\begin{equation*}
  \varphi(n) = \begin{cases} a, & n \textrm{ odd},\\ b,& n \textrm { even}
\end{cases}
\end{equation*}
The corresponding total phase space map $\PP \varphi: M\times N\to (M\times N)^n$
is given by the formula
\[
\PP (x,y) = (x,y,x,y,\ldots, x,y).
\]
The groupoid $\GG (G',\cP')$ is trivial. Consequently $\V
(G',\cP')$ consists of a pair of control systems $w'_a: M\times N \to
TM$ and $w'_b: N\times M \to TN$.  They interconnect to define a
vector field $\scI' (w'):M\times N \to TM\times TN$ with 
\[
\scI' (w') (x,y) = (w'_a (x,y), w'_b (y,x))
\] 
for all $(x,y) \in M\times N = \PP G'$.

The groupoid $\GG (G,\cP)$ is {\em not} trivial: all input networks
corresponding to odd numbered vertices are uniquely isomorphic.  That
is, given the vertices $2k+1$ and $2\ell +1$, $k\not = \ell$, there is
exactly one isomorphism $\psi_{k\ell}:I(2k+1)\to I(2\ell +1)$ as well
$\psi_{\ell k} = {\psi_{k\ell}}\inv: I(2\ell +1) \to I(2k+1)$. Analogous
statement holds for input networks corresponding to even numbered
vertices.  

The pullback map 
\[
\varphi^*: \V (G',\cP')\to \V (G,\cP)
\]
is easily seen to be given by
\[
\varphi^* (w'_a,w'_b) = (w'_a, w'_b,\ldots, w'_a, w'_b),
\]
and $\scI (\varphi^*w')\in \Gamma T(M\times N)^n$ is given by
\[
\scI (\varphi^*w') (x_1,y_1,\ldots, x_{n}, y_n) =
(w'_a (x_1,y_1), w'_b (y_1,x_1), \ldots, w'_a (x_n,y_n), w'_b (y_n,x_n)).
\]
It is clear that $\PP \varphi (M\times N)$ is an invariant submanifold
of the vector field $\scI (\varphi^*w')$, as should be expected in 
light of Theorem~\ref{thm:main}.
\end{example}

\subsection{Injective fibrations}
Consider an injective fibration $\varphi: (G,\cP)\to (G',\cP')$ of
networks.  Lemma~\ref{lem:5.2.1} below shows that the map $\PP
\varphi: \PP G'\to \PP G$ of total phase spaces is a surjective
submersion.  Combining this with Theorem~\ref{thm:main} we see that for any groupoid-invariant virtual vector field $w'\in \V (G',\cP')$ the map
\[
\PP \varphi: (\PP G', \scI' (w'))\to (\PP G, \scI (\varphi^*w'))
\]
is a projection of dynamical systems.  In particular for any singular
point $x$ of the vector field $\scI (\varphi^*w')$, i.e., the point
where the vector field is zero, the fiber $\PP \varphi\inv (x)$ is an
invariant submanifold of the vector field $\scI (w')$.

Note also that since the map of graphs $\varphi: G\to G'$ is injective,
$\varphi: G\to \varphi(G)$ is an isomorphism.  Since $\varphi$ is also
a graph fibration, there are no edges of $G'$ with the source in
$G'\smallsetminus \varphi (G)$ and target in the image $\varphi(G)$.
Thus the image of $\varphi$ is a subsystem of $G'$ that drives the
dynamical system on $G'$. In other words the notion of an injective
fibration makes precise the intuitive idea of a subsystem driving a
larger network.

\begin{lemma}\label{lem:5.2.1}
  Suppose $\varphi:(G,\cP)\to (G',\cP')$ is an injective fibration.
  Then $\PP\varphi: \PP G'\to \PP G$ is a surjective submersion.
\end{lemma}
\begin{proof}
  Since $\varphi:G\to G'$ is injective, the set of nodes $G_0'$ of
  $G'$ can be partitioned as the disjoint union of the image
  $\varphi(G_0)$, which is a copy of $G_0$, and the complement.  Hence
\[
\PP G' \simeq \bigsqcap_{a\in G_0} \cP(\varphi(a)) \times
\bigsqcap_{a'\not\in \varphi(G_0)} \cP' (a') \simeq \PP G \times
\bigsqcap_{a'\not\in \varphi(G_0)} \cP' (a').
\]
With respect to this identification of $\PP G'$ with $\PP G \times
\bigsqcap_{a'\not\in \varphi(G_0)} \cP' (a')$ the map $\PP\varphi:\PP
G'\to \PP G$ is the projection
\[
 \PP G \times
\bigsqcap_{a'\not\in \varphi(G_0)} \cP' (a') \to \PP G.
\]
which is a surjective submersion.
\end{proof}
\begin{example}
Consider the injective graph fibration
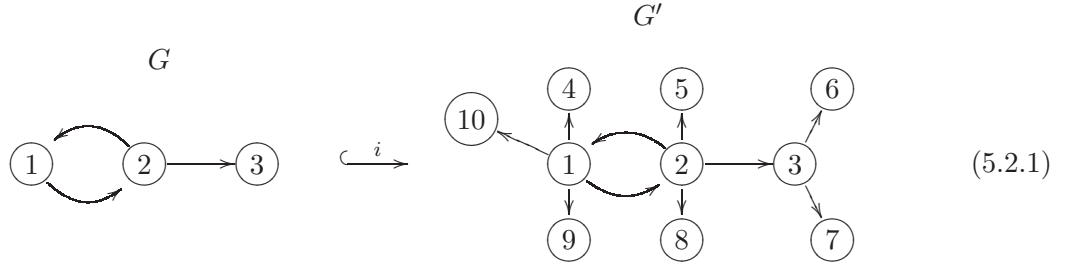
\begin{equation}\label{eq:ten}
\xy
(15,14)*+{G};
(80,20)*+{G'};
(13,0)*++{ 
\xy
(-6,0)*++[o][F]{1}="1"; 
(9,0)*++[o][F]{2}="2"; 
(24,0)*++[o][F]{3}="3"; 
{\ar@/_1.2pc/ "1";"2"};
{\ar@{->}_{} "2";"3"};
{\ar@/_1.2pc/ "2";"1"};
\endxy
}="2";
(80, 0)*++{
\xy
(-45,0)*++[o][F]{1}="1"; 
(-30,0)*++[o][F]{2}="2"; 
(-15,0)*++[o][F]{3}="3"; 
(-58,6)*++[o][F]{10}="10";
(-45,10)*++[o][F]{4}="4"; 
(-30,10)*++[o][F]{5}="5"; 
(-10,10)*++[o][F]{6}="6"; 
(-10,-10)*++[o][F]{7}="7";
(-45,-10)*++[o][F]{9}="9"; 
(-30,-10)*++[o][F]{8}="8";  
{\ar@{->}@/_1pc/ "2";"1"};
{\ar@/_1pc/ "1";"2"};
{\ar@{->} "2";"3"};
{\ar@{->}_{} "3";"6"};
{\ar@{->}_{} "3";"7"};
{\ar@{->}_{} "2";"5"};
{\ar@{->}_{} "2";"8"};
{\ar@{->}_{} "1";"4"};
{\ar@{->}_{} "1";"9"};
{\ar@{->}_{} "1";"10"};
\endxy
}="3";
  {\ar@{^{(}->}^{i} (40,0); (48, 0)}; 
\endxy
\end{equation}
Choose phase space functions
$\cP, \cP'$ so that $i: (G,\cP)\to (G',\cP')$ is a map of
networks. By the discussion above, for any choice of a groupoid
invariant virtual vector field $w'\in \V (G',\cP')$ the dynamics in
the subsystem $(\PP G, \scI(i^*w'))$ drives the entire system
$(\PP G', \scI (w'))$.  This is intuitively clear from the
graph~\eqref{eq:ten} since there are no ``feedbacks'' from vertices
$4, \ldots,10$ back into $1,2,3$.%
\end{example}

\subsection{General maps}

As we observed in the beginning of the section any fibration
$\varphi:(G,\cP)\to (G',\cP')$ can be factored as a surjection onto
its image 
\[
\varphi: (G,\cP)\to (\varphi(G),\cP')
\]
followed by the inclusion 
\[
i:
(\varphi(G),\cP')\hookrightarrow (G',\cP').
\]
It follows from the two subsections above that for any groupoid
invariant virtual vector field $w'\in \V (G',\cP')$ the map of
dynamical systems
\[
 \PP \varphi: (\PP G', \scI'(w')) \to (\PP G, \scI (\varphi^*w'))
\]
factors as a projection of dynamical systems
\[
\PP i:  (\PP G', \scI'(w')) \twoheadrightarrow (\PP\varphi(G), \scI' (i^*w'))
\]
followed by the embedding
\[
 (\PP\varphi(G), \scI' (i^*w')) \to (\PP G, \scI (\varphi^*w')).
\]
\begin{example}
Consider the graph fibration
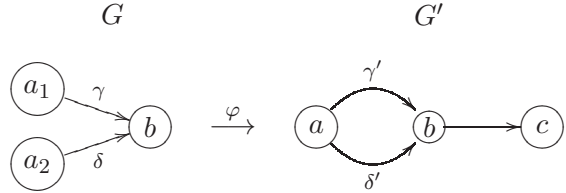
\[ 
 \xy
(-5,15)*++{G};
(-15,5)*++[o][F]{a_1}="1";
(-15,-5)*++[o][F]{a_2}="1'"; 
(0,0)*++[o][F]{b}="2"; 
{\ar@{->}^{\gamma} "1";"2"};
{\ar@{->}_{\delta} "1'";"2"};
\endxy
\quad \stackrel{\varphi}{\longrightarrow} \quad
\xy
(-15,0)*++[o][F]{a}="1"; 
(0,0)*+[o][F]{b}="2";
 (15,0)*++[o][F]{c}="5";
(0,15)*+{G'};
{\ar@/_1.2pc/_{\delta'} "1";"2"};
{\ar@/^1.2pc/^{\gamma'} "1";"2"};
{\ar@{->}^{} "2";"5"};
\endxy
\] 
from Example~\ref{example:5}.  Choose a phase space function $\cP'$ on
$G'$ and define $\cP:G_0\to \Man$ by $\cP(a_1) = \cP(a_2) = \cP'(a)$,
$\cP(b) = \cP'(b)$.  Then $\varphi: (G,\cP)
\to (G',\cP')$ is a fibration of networks.  It factors as
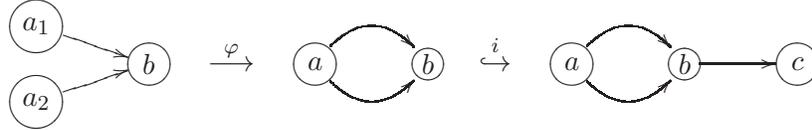
\[
 \xy
(-15,5)*++[o][F]{a_1}="1";
(-15,-5)*++[o][F]{a_2}="1'"; 
(0,0)*++[o][F]{b}="2"; 
{\ar@{->}^{} "1";"2"};
{\ar@{->}_{} "1'";"2"};
\endxy
\quad \stackrel{\varphi}{\longrightarrow} \quad
\xy
(-15,0)*++[o][F]{a}="1"; 
(0,0)*+[o][F]{b}="2";
{\ar@/_1.2pc/_{} "1";"2"};
{\ar@/^1.2pc/^{} "1";"2"};
\endxy
\quad \stackrel{i}{\hookrightarrow} \quad
\xy
(-15,0)*++[o][F]{a}="1"; 
(0,0)*+[o][F]{b}="2";
 (15,0)*++[o][F]{c}="5";
{\ar@/_1.2pc/_{} "1";"2"};
{\ar@/^1.2pc/^{} "1";"2"};
{\ar@{->}^{} "2";"5"};
\endxy
\] 
\end{example}

\section{Spaces of invariant virtual vector fields}\label{sec:6}

The purpose of this section is to characterize further and more
precisely the space of groupoid-invariant virtual vector fields on a
network $(G,\cP)$ and to understand better the pullback maps
$\varphi^*:\V (G',\cP')\to (G,\cP)$ induced by fibrations of networks
$\varphi:(G,\cP)\to (G',\cP')$.

\subsection{The space $\V(G,\cP)$ as a product of  spaces of fixed vectors}
It will be useful to introduce a bit more notation.

\begin{notation}
  Given a network $(G,\cP)$ we have an evident bijection between the
  set $G_0$ of vertices of the graph $G$ and the set $\GG_0=\{(I(a),
  \cP\circ \xi_a)\}_{a\in G_0}$ of objects of the groupoid $\GG
(G,\cP)$.  It will be convenient to identify the two sets:
\[
\GG (G,\cP)_0 = G_0.
\]
\end{notation}

\begin{definition}[Automorphism group] For a vertex $a$ of a graph
  $G$, hence for an object of the symmetry groupoid $\GG (G,\cP)$ of a
  network $(G,\cP)$ we set
\[
\Aut (a):= \{ \psi: (I(a),\cP\circ \xi_a)\to (I(a),\cP\circ \xi_a) \mid \psi \textrm{ is an isomorphism of networks }\}
\]
Clearly $\Aut(a)$ is a group under composition.  We call it the {\sf
  automorphism group} of the vertex $a$.
\end{definition}

\begin{remark}\label{rmrk:6.1.6}
  Note that $\Aut(a)$ is the collection of isomorphisms of the
  groupoid $\GG(G,\cP)$ with source and target $a$.

  By construction $\Aut(a)$ acts on the vector space $\Ctrl(\PP
  I(a)\to \PP a)$, the fiber of the bundle $\Control(G,\cP)\to G_0$.
  We denote the space of fixed vectors by $\Ctrl(\PP I(a)\to \PP
  a)^{\Aut(a)}$
\end{remark}

\begin{remark}
  In general given an object $a$ of a groupoid $\HH =\{\HH_1\toto
  \HH_0\}$ we have a group $\Aut(a)$ consisting of isomorphism of
  $\HH$ with source and target $a$.
\end{remark}

\begin{definition}
  Given a groupoid $\HH$ we say that two objects $a$ and $b$ of $\HH$
  are {\sf isomorphic} if there is an isomorphism $\gamma$ of $\HH$
  with source $a$ and target $b$.
\end{definition}

\begin{remark}
  It follows easily from the definition of a groupoid that being
  isomorphic is an equivalence relation on the objects.  We denote the
  collection of isomorphism classes of objects of a groupoid $\HH$ by
  $\HH_0/\HH_1$ and denote the isomorphism class of an object $a$ by $[a]$.
\end{remark}

\begin{lemma}
  Let $(G,\cP)$ be a network.  The space $\V(G,\cP)$ of groupoid
  invariant virtual vector fields is isomorphic (as a vector space) to
  the product
\[
\bigsqcup_{[a]\in \GG_0/\GG_1} \Ctrl(\PP I(a)\to \PP
  a)^{\Aut(a)}.
\]
Here as in Remark~\ref{rmrk:6.1.6}\, $\Ctrl(\PP I(a)\to \PP a)^{\Aut(a)}$
is the space of vectors fixed by the action of $\Aut(a)$.
\end{lemma}
\begin{proof}
  Suppose $w\in \V(G,\cP)$ is an invariant section of
  $\Control(G,\cP)\to G_0$.  Then for any node $a$ of $G_0$ and any
  automorphism $\psi\in \Aut(a)$ we have
\[
\Ctrl(\psi)w_a = w_a.
\]
Hence $w_a\in \Ctrl(\PP I(a)\to \PP a)^{\Aut(a)}$.  

If $a$ and $b$ are two isomorphic objects in the groupoid $\GG
(G,\cP)$ by way of $\psi: (I(a),\cP\circ \xi_a)\to (I(b),\cP\circ
\xi_b)$ then 
\[
w_b = \Ctrl(\psi) w_b.
\]
It follows that if we pick representatives $a_1,\cdots, a_N\in G_0$ of
the equivalence classes in $\GG_0/\GG_1$ then the restriction map
\[
\V (G,\cP) \to \bigsqcup_{i=1} ^N\Ctrl(\PP I(a_i)\to \PP
  a_i)^{\Aut(a_i)}, \quad w\mapsto (w_{a_1},\ldots, w_{a_N}) 
\]
is an isomorphism of vector spaces.  This proves the lemma.
\end{proof}

\begin{example} \label{ex:6.1.8}
Consider the network $(G,\cP)$ where $G$ is the graph
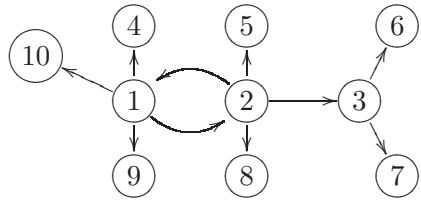
\[
\xy
(-45,0)*++[o][F]{1}="1"; 
(-30,0)*++[o][F]{2}="2"; 
(-15,0)*++[o][F]{3}="3"; 
(-58,6)*++[o][F]{10}="10";
(-45,10)*++[o][F]{4}="4"; 
(-30,10)*++[o][F]{5}="5"; 
(-10,10)*++[o][F]{6}="6"; 
(-10,-10)*++[o][F]{7}="7";
(-45,-10)*++[o][F]{9}="9"; 
(-30,-10)*++[o][F]{8}="8";  
{\ar@{->}@/_1pc/ "2";"1"};
{\ar@/_1pc/ "1";"2"};
{\ar@{->} "2";"3"};
{\ar@{->}_{} "3";"6"};
{\ar@{->}_{} "3";"7"};
{\ar@{->}_{} "2";"5"};
{\ar@{->}_{} "2";"8"};
{\ar@{->}_{} "1";"4"};
{\ar@{->}_{} "1";"9"};
{\ar@{->}_{} "1";"10"};
\endxy
\]
 and $\cP$ assigns the same manifold $M$ to each vertex of $G$.
Then the input trees of $G$ are all  of the form
\[
\xy
(-45,0)*++[o][F]{a}="1"; 
(-30,0)*++[o][F]{b}="2"; 
{\ar@{->} "1";"2"};
\endxy
\]
and the corresponding input networks are all isomorphic.  Moreover
they have trivial automorphism groups. Consequently
\[
\V(G,\cP)\simeq \Ctrl(M\times M\to TM).
\]
This is quite small compared to the space of all vector fields on the
total phase space $\PP G \simeq M^{10}$.
\end{example}

\subsection{Maps between spaces of invariant virtual vector fields}
The goal of this subsection is to understand the pullback map
$\varphi^*: \V (G',\cP')\to \V (G,\cP)$ between groupoid-invariant
virtual vector fields induced by a fibration $\varphi:(G,\cP)\to
(G',\cP)$ of networks.  We will see that $\varphi^*$ is always
surjective.  To describe its kernel we need the following concept.

\begin{definition}
  Let $\varphi:(G,\cP)\to (G',\cP)$ be a fibration of networks.  The
  {\sf essential image} $\essim \varphi \subset G_0'$ of $\varphi$
  consists of all the vertices $a'\in G_0'$ so that there is an
  isomorphism
\[
\psi: (I(a'),\cP'\circ\xi_{a'})\to (I(\varphi(a)), \cP\circ \xi_{\varphi(a)})
\]
of input networks for some vertex $a$ of $G$.

We say that $\varphi$ is {\sf essentially surjective} if $\essim
\varphi = G_0'$.
\end{definition}

\begin{example}\label{ex:6.2.2}
 The map 
\[
\xy
(15,14)*+{G};
(80,20)*+{G'};
(13,0)*++{ 
\xy
(-6,0)*++[o][F]{1}="1"; 
(9,0)*++[o][F]{2}="2"; 
(24,0)*++[o][F]{3}="3"; 
{\ar@/_1.2pc/ "1";"2"};
{\ar@{->}_{} "2";"3"};
{\ar@/_1.2pc/ "2";"1"};
\endxy
}="2";
(80, 0)*++{
\xy
(-45,0)*++[o][F]{1}="1"; 
(-30,0)*++[o][F]{2}="2"; 
(-15,0)*++[o][F]{3}="3"; 
(-58,6)*++[o][F]{10}="10";
(-45,10)*++[o][F]{4}="4"; 
(-30,10)*++[o][F]{5}="5"; 
(-10,10)*++[o][F]{6}="6"; 
(-10,-10)*++[o][F]{7}="7";
(-45,-10)*++[o][F]{9}="9"; 
(-30,-10)*++[o][F]{8}="8";  
{\ar@{->}@/_1pc/ "2";"1"};
{\ar@/_1pc/ "1";"2"};
{\ar@{->} "2";"3"};
{\ar@{->}_{} "3";"6"};
{\ar@{->}_{} "3";"7"};
{\ar@{->}_{} "2";"5"};
{\ar@{->}_{} "2";"8"};
{\ar@{->}_{} "1";"4"};
{\ar@{->}_{} "1";"9"};
{\ar@{->}_{} "1";"10"};
\endxy
}="3";
  {\ar@{^{(}->}^{i} (40,0); (48, 0)}; 
\endxy
\]
of networks is not surjective.  But it is essentially surjective if
$\cP'(i) = \cP'(j)$ for all $1\leq i< j\leq 10$, i.e., if we assign
the same manifold to all vertices of the graphs.
\end{example}

\begin{theorem}
  Let $\varphi:(G,\cP)\to (G',\cP')$ be a fibration of networks.  Then
  $\varphi^*: \V (G',\cP')\to \V (G,\cP)$ is surjective.  The kernel
  of $\varphi^*$ is the space
\[
\ker \varphi^* = \{w'\in \V (G',\cP')\mid w'_{a'}=0 \textrm{ for all } a'\in \essim \varphi\},
\]
where $\essim \varphi$ is the essential image of $\varphi$ defined
above.
\end{theorem}
\begin{proof}
  For $w'\in \V (G',\cP')$ the pullback $\varphi^*w'$ is zero if and
  only if the component $(\varphi^*w')_a = 0$ for all $a\in G_0$.
  Since $(\varphi^*w')_a =\Ctrl(\varphi_a)\inv w_{\varphi(a)}$ (q.v.\
  proof of Proposition~\ref{prop:4.2.1}) and since
  $\Ctrl(\varphi_a)\inv$ is an isomorphism we conclude that
\[
\varphi^*w'= 0 \quad \Leftrightarrow \quad w'_{\varphi(a)} \textrm{ for all } a\in G_0.
\]
Finally note that an invariant section $w'\in \V(G',\cP')$ vanishes on
the image of $\varphi$ if and only if it vanishes on the essential
image of $\varphi$.
\end{proof}

\begin{corollary}
  If $\varphi:(G,\cP)\to (G',\cP')$ is an essentially surjective
  fibration of networks then $\varphi^*:\V(G',\cP')\to \V (G,\cP)$ is
  an isomorphism.  In particular $\varphi^*$ is an isomorphism if
  $\varphi$ is surjective.
\end{corollary}

\begin{example} Consider the map $i$ of networks in
  Example~\ref{ex:6.2.2}. Since $i$ is injective and essentially
  surjective the map $i^*: \V(G',\cP')\to \V (G,\cP)$ is an
  isomorphism.  Compare with Example~\ref{ex:6.1.8}.  Clearly the map
  $i$ is very far from being surjective.
\end{example}

\begin{remark} As we pointed out in Remark~\ref{rem:GS}
  in the groupoid formalism of Golubitsky {\em et
    al.}
  the quotient maps defined by
  balanced equivalence relations are surjective.  Hence the spaces of
  groupoid invariant vector fields on a network and on its quotient by
  a balanced equivalence relation are always isomorphic.

\end{remark}

\section*{Acknowledgments}

The authors thank the anonymous referees for many valuable comments that lead to significant improvements in the paper.  L.D. was partially
supported by NSF grant CMG-0934491.


\end{document}